\documentclass[12pt, a4paper]{article}

\usepackage[ngerman,english]{babel}	

\usepackage{titling}
\usepackage{authblk} 

\usepackage{graphicx}
\usepackage{amsmath} 
\usepackage{amssymb} 
\usepackage{amsfonts}
\usepackage{amsthm}
\usepackage{enumerate}
\usepackage{makeidx} 
\usepackage{verbatim}
\usepackage{epsfig}
\usepackage{dsfont}
\usepackage{wasysym} 
\usepackage{caption}	
\usepackage{float}
\usepackage{stackrel}	
\usepackage{mathtools} 
\usepackage{ltablex}	
\usepackage{stmaryrd}	

\usepackage{tikz}
\usepackage{pgfplots}

\usepackage[backref=page]{hyperref}
\hypersetup{colorlinks=true,
						linkcolor=black,
						citecolor=black,
						urlcolor=blue}

\makeindex

\makeindex      
     
\newlength{\fixboxwidth}     
\newcommand{\R}{{\mathbb R}}

\newcommand{\N}{{\mathbb N}}
\newcommand{\Z}{{\mathbb Z}}

\renewcommand{\P}{\mathbb P}
\newcommand\expect{\operatorname{\mathbb{E}}}

\DeclareMathOperator\Uniform{unif}

\DeclareMathOperator\median{med}

\newcommand{\dint}{\,\mathup{d}}

\newcommand\mon{\textup{mon}}

\newcommand{\ind}{\mathds{1}}

\DeclareMathOperator\sgn{sgn}

\newcommand\ran{\textup{ran}}
\newcommand\deter{\textup{det}}

\newcommand\nonada{\textup{nonada}}
\newcommand\ada{\textup{ada}}

\DeclareMathOperator\App{APP}
\DeclareMathOperator\Int{INT}

\newcommand\eps{\varepsilon}
\newcommand\euler{\mathup{e}}

\newcommand\fracts[2]{{\textstyle\frac{#1}{#2}}}

\newcommand{\zeros}{\boldsymbol{0}}
\newcommand{\ones}{\boldsymbol{1}}

\newcommand\vece{\mathbf{e}}

\newcommand\veci{\mathbf{i}}

\newcommand\vecu{\mathbf{u}}
\newcommand\vecv{\mathbf{v}}
\newcommand\vecw{\mathbf{w}}
\newcommand\vecx{\mathbf{x}}
\newcommand\vecy{\mathbf{y}}
\newcommand\vecz{\mathbf{z}}
\newcommand\vecX{\mathbf{X}}

\newcommand\vecZ{\mathbf{Z}}

\newcommand\vecalpha{\boldsymbol{\alpha}}
\newcommand\vecbeta{\boldsymbol{\beta}}
\newcommand\vecdelta{\boldsymbol{\delta}}
\newcommand\veckappa{\boldsymbol{\kappa}}
\newcommand\veclambda{\boldsymbol{\lambda}}

\DeclareMathAlphabet{\mathup}{OT1}{\familydefault}{m}{n}

\usepackage{microtype} 

\usepackage{bookmark}
\makeatletter 
  \providecommand*{\toclevel@author}{999}
  \providecommand*{\toclevel@title}{0}
\makeatother

\theoremstyle{plain}       
     
\newtheorem{theorem}{Theorem}[section]      
      
\newtheorem{lemma}[theorem]{Lemma}     
\newtheorem{proposition}[theorem]{Proposition}      
     
\theoremstyle{definition}

\newtheorem{remark}[theorem]{Remark}

\numberwithin{theorem}{section}

\newcounter{proofstep}
\newcommand{\proofstep}[2]{\refstepcounter{proofstep} \label{#1}
	\subsubsection{Step~\arabic{proofstep}: #2}}
\newcounter{proofsubstep}[proofstep]
\newcommand{\proofsubstep}[2]{\refstepcounter{proofsubstep} \label{#1}
												\textbf{\textit{Step~\arabic{proofstep}.\arabic{proofsubstep}: #2 \\}}}

\newcommand{\proofstepref}[1]{\ref{#1}}

\newcommand{\thmref}[1]{\hyperref[#1]{Theorem~\ref*{#1}}}
\newcommand{\lemref}[1]{\hyperref[#1]{Lemma~\ref*{#1}}}
\newcommand{\propref}[1]{\hyperref[#1]{Proposition~\ref*{#1}}}
\newcommand{\corref}[1]{\hyperref[#1]{Corollary~\ref*{#1}}}
\newcommand{\remref}[1]{\hyperref[#1]{Remark~\ref*{#1}}}
\newcommand{\chapref}[1]{\hyperref[#1]{Chapter~\ref*{#1}}}
\newcommand{\secref}[1]{\hyperref[#1]{Section~\ref*{#1}}}
\newcommand{\defref}[1]{\hyperref[#1]{Definition~\ref*{#1}}}
\newcommand{\exref}[1]{\hyperref[#1]{Example~\ref*{#1}}}

\definecolor{darkgrey}{rgb}{0.4, 0.4, 0.4}
\definecolor{lightgrey}{rgb}{0.75,0.75,0.75}

\begin{document}

\title{The Difficulty of Monte Carlo Approximation
	of Multivariate Monotone Functions}
\author{Robert J. Kunsch\thanks{E-mail: robert.kunsch@uni-osnabrueck.de}}
\affil{Universit\"at Osnabr\"uck, Institut f\"ur Mathematik,
	Albrechtstr.~28a, 49076 Osnabr\"uck, Germany}
\date{\today}
\maketitle

\begin{abstract}
	We study the $L_1$-approximation of $d$-variate monotone functions
	based on information from~$n$ function evaluations.
	It is known that this problem suffers from the curse of dimensionality
	in the deterministic setting,
	that is, the number $n(\eps,d)$ of function evaluations needed
	in order to approximate an unknown monotone function
	within a given error threshold~$\eps$
	grows at least exponentially in~$d$.
	This is not the case in the randomized setting (Monte Carlo setting)
	where the complexity~\mbox{$n(\eps,d)$}
	grows exponentially in~$\sqrt{d}$ (modulo logarithmic terms) only.
	An algorithm exhibiting this complexity is presented.
	Still, the problem remains difficult as best known methods are deterministic
	if~$\eps$ is comparably small, namely~$\eps \preceq 1/\sqrt{d}$.
	This inherent difficulty is confirmed by lower complexity bounds
	which reveal a joint~$(\eps,d)$-dependency and from which we deduce
	that the problem is \emph{not weakly tractable}.
\end{abstract}

\noindent
\textbf{Keywords.\;} Monte Carlo approximation;
monotone functions;
information-based complexity;
standard information;
intractable;
curse of dimensionality.

\section{Introduction} \label{sec:intro}

Within this paper we consider
the $L_1$-approximation of $d$-variate \emph{monotone}
functions using function values as information,
\begin{equation*}
	\App : F_{\mon}^d \hookrightarrow L_1([0,1]^d) \,,\;
	f \mapsto f \,,
\end{equation*}
where the input set
\begin{equation*}
	F_{\mon}^d := \{ f: [0,1]^d \rightarrow [-1,1] \mid
						\vecx \leq \tilde{\vecx} \Rightarrow f(\vecx) \leq f(\tilde{\vecx})\}
\end{equation*}
consists of monotonically increasing functions
with respect to the partial order on the domain.
For~\mbox{$\vecx,\tilde{\vecx} \in \R^d$}, the partial order is defined by
\begin{equation*}
	\vecx \leq \tilde{\vecx}
		\qquad:\Leftrightarrow\qquad
			x_j \leq \tilde{x}_j \text{ for all } j=1,\ldots,d \,.
\end{equation*}

Approximation of monotone functions is not a linear problem
as defined in the book on \emph{Information-based Complexity} (IBC)
by Traub et al.~\cite{TWW88},
because the set~$F_{\mon}^d$
is not symmetric:
For non-constant functions~$f \in F_{\mon}^d$,
the negative~$-f$ is not contained in~$F_{\mon}^d$
since it will be monotonically decreasing.
The monotonicity assumption is different
from common smoothness assumptions, yet it implies many other nice properties,
see for example Alberti and Ambrosio~\cite{AA99}.
Integration and approximation of monotone functions have been studied
in several papers~\cite{HNW11,No92mon,Papa93} to which we will refer
in the course of this paper.
Monotonicity can also be an assumption for statistical problems~\cite{GW07,RM51}.
A similar structural assumption could be convexity (more general: $k$-monotonicity),
numerical problems with such properties have been studied for example
in~\cite{CDV09,HNW11,KNP96,Kop98,NP94}.

For the problem of approximating monotone functions
with respect to the $L_1$-norm,
a deterministic algorithm is a mapping
\begin{equation*}
	A_n: F_{\mon}^d \xrightarrow{N} \R^n \xrightarrow{\phi} L_1([0,1]^d) \,,
\end{equation*}
where~$N$ is the \emph{information mapping}
\begin{equation*}
	N(f) = (y_1,\ldots,y_n)
		:= (f(\vecx_1),\ldots,f(\vecx_n)) \,.
\end{equation*}
The nodes $\vecx_1,\ldots,\vecx_n$ may be selected in an adaptive manner,
that is, the choice of the node~$\vecx_i$
may depend on previously obtained information~$y_1,\ldots,y_{i-1}$.
(One could even vary the number~$n$ of computed function values
in an adaptive way,
thereby building an algorithm with so-called \emph{varying cardinality}.)
The error of such a method is defined by the worst case,
\begin{equation*}
	e(A_n,F_{\mon}^d)
		:= \sup_{f \in F_{\mon}^d} \|A_n(f) - f\|_{L_1} \,.
\end{equation*}
A Monte Carlo method~$A_n = (A_n^{\omega})_{\omega \in \Omega}$
is a family of such mappings indexed by a random element~$\omega$
from a probability space~$(\Omega,\Sigma,\P)$,
hence for a fixed input~$f$
the realization~$A_n^{\omega}(f)$ is a random variable
with values in~$L_1([0,1]^d)$.
We assume sufficient measurability such that the error,
\begin{equation*}
	e((A_n^{\omega}),F_{\mon}^d)
		:= \sup_{f \in F_{\mon}^d} \expect \|A_n(f) - f\|_{L_1} \,,
\end{equation*}
is well defined.
We aim to compare the deterministic setting with the Monte Carlo setting
in terms of the minimal error achievable with
an information budget~$n \in \N_0$,
\begin{equation*}
	e^{\deter}(n,F_{\mon}^d)
		:= \inf_{A_n} e(A_n,F_{\mon}^d)
	\qquad \text{vs.} \qquad
	e^{\ran}(n,F_{\mon}^d)
		:= \inf_{(A_n^{\omega})} e((A_n^\omega),F_{\mon}^d) \,,
\end{equation*}
or the complexity for a given error threshold~$\eps > 0$
and dimension~$d \in \N$,
\begin{align*}
	n^{\deter}(\eps,F_{\mon}^d)
		&:= \inf \{n \in \N_0 \mid
							\exists A_n \colon e(A_n,F_{\mon}^d) \leq \eps\}\\
	\text{vs.} \qquad
	n^{\ran}(\eps,F_{\mon}^d)
		&:= \inf \{n \in \N_0 \mid
							\exists (A_n^{\omega}) \colon
								e((A_n^{\omega}),F_{\mon}^d) \leq \eps\} \,.
\end{align*}
Section~\ref{sec:Simple} is a collection of results that
-- if not directly stated in previous papers --
can be shown with well known techniques.
In Section~\ref{sec:Order}
we show that for fixed dimension~$d$ the order of convergence
for the $L_1$-approximation of monotone functions
cannot be improved by randomization.
The curse of dimensionality for deterministic approximation,
see Hinrichs, Novak and Wo\'zniakowski~\cite{HNW11},
is recalled in Section~\ref{sec:Curse}.
Section~\ref{sec:UBs} contains the good news that the curse is broken
by Monte Carlo methods, namely, we show the combined upper bound
\begin{align*}
	n^{\ran}(\eps,F_{\mon}^d)
		\,\leq\, \min\Biggl\{&
			\exp\left[C \, \frac{\sqrt{d}}{\eps}
									\, \left( 1 + \log \frac{d}{\eps}\right)^{3/2}
					\right],\,
			\exp\left[d \, \log \frac{d}{2 \eps}\right]
			\Biggr\} \,,
\end{align*}
with some numerical constant~$C > 0$, see \thmref{thm:monoUBs}.
Here, the first bound is achieved by a proper Monte Carlo method
and applies in the pre-asymptotic regime,
whereas the second bound is achieved by a deterministic algorithm
and applies for small error thresholds~\mbox{$\eps \preceq 1/\sqrt{d}$} (modulo logarithmic terms).
Lower bounds in the Monte Carlo setting are found in Section~\ref{sec:LBs},
we prove
\begin{equation*}
	n^{\ran}(\eps,F_{\mon}^d)
		\,>\, \nu \, \exp\left[c \, \frac{\sqrt{d}}{\eps}\right] \,,
	\qquad \text{for \mbox{$\eps_0 \, \sqrt{d_0/d} \leq \eps \leq \eps_0$}
		and \mbox{$d \geq d_0$},}
\end{equation*}
with numerical constants $\nu,c > 0$, see \thmref{thm:LB}.
There is a constraint on~$\eps$,
which is not surprising as it fits to the observation
that for smaller~$\eps$ best known algorithms are deterministic
and we have a different joint $(\eps,d)$-dependency in that regime.
However, by monotonicity of the $\eps$-complexity,
we can still conclude
\begin{equation*}
	n^{\ran}(\eps,F_{\mon}^d)
		\,>\, \nu \, \exp\left[c' \, d\right] \,,
	\qquad \text{for~$0 < \eps \leq \eps_0 \, \sqrt{d_0/d}$
			and \mbox{$d \geq d_0$},}
\end{equation*}
where~$c' = c/(\eps_0\sqrt{d_0})$.
Hence, the lower bounds match the upper bounds except for logarithmic terms
in the exponent.
The bad news is:
For moderately decaying error thresholds $\eps = \eps_0 \sqrt{d_0/d}$,
the Monte Carlo complexity depends already exponentially on~$d$,
we conclude that the problem is \emph{not weakly tractable},
see \remref{rem:monMCLBintractable}.

This paper is concerned
with real-valued monotone functions~\mbox{$f:[0,1]^d \to [-1,1]$}.
A closely related problem is
the approximation of \emph{Boolean} monotone functions
\mbox{$f: \{0,1\}^d \rightarrow \{0,1\}$}.
The algorithm we present in Section~\ref{sec:UBs}
is inspired by an approximation method for Boolean monotone functions
due to Bshouty and Tamon~\cite{BT96}.
The Monte Carlo lower bounds given in Section~\ref{sec:LBs}
are actually obtained by a reduction
to the approximation of Boolean monotone functions.
It is then a modification of a lower bound proof
which can be found in Blum, Burch and Langford~\cite{BBL98}.
Similarly to the real-valued setting in \secref{sec:Curse},
one can show the curse of dimensionality for deterministic approximation
of Boolean monotone functions as well,
see the author's PhD thesis~\cite[Theorem~4.5]{Ku17}.
So even for the simpler problem we can state
that Monte Carlo breaks the curse.
The main difference to real-valued monotone functions is
that the concept of order of convergence,
see~\secref{sec:Order}, is meaningless
for a discrete problem such as the approximation of Boolean functions.

\section{Survey on deterministic approximation}
\label{sec:Simple}

\subsection{The classical approach -- order of convergence}
\label{sec:Order}

The classical approach for the numerical analysis of multivariate problems
is to fix the dimension~$d$ and to study the order at which the error
$e(n)$ converges to zero as the information budged~$n$ grows.
We will use the common asymptotic notation
\begin{equation*}
	a_n \preceq e(n) \preceq b_n
	\qquad:\Leftrightarrow\qquad
	\exists\, c,C > 0 \colon \; c \, a_n \leq e(n) \leq C \, b_n \,.
\end{equation*}
If $a_n \preceq e(n) \preceq a_n$, we simply write $a_n \asymp e(n)$.
The hidden constants~$c$ and $C$ may depend on problem parameters
such as the dimension.
Sometimes this dependency is in a very unpleasant way.

As an example, the order of convergence has been studied for the problem of
approximating the integral of monotone functions,
\begin{equation*}
	\Int: F_{\mon}^d \rightarrow \R, \quad
					f \mapsto \int_{[0,1]^d} f \dint\vecx \,,
\end{equation*}
based on finitely many function evaluations.
Interestingly, for this problem
adaption makes a difference in the randomized setting
(at least for~\mbox{$d = 1$}),
but non-adaptive randomization helps only for~\mbox{$d \geq 2$}
to speed up the convergence compared to deterministic methods.
In the univariate case Novak~\cite{No92mon} showed
\begin{multline*}
	e^{\ran,\ada}(n,\Int,F_{\mon}^1) \asymp n^{-3/2} \\
		\prec e^{\ran,\nonada}(n,\Int,F_{\mon}^1)
			\asymp e^{\det}(n,\Int,F_{\mon}^1) \asymp n^{-1} \,.
\end{multline*}
Papageorgiou~\cite{Papa93} examined
the integration of $d$-variate monotone functions,
for dimensions~\mbox{$d \geq 2$} we have
\begin{multline*}
	e^{\ran,\ada}(n,\Int,F_{\mon}^d)
		\asymp n^{-1/d - 1/2} \\
	\preceq
	e^{\ran,\nonada}(n,\Int,F_{\mon}^d)
		\preceq n^{- 1/(2d) - 1/2}\\
	\prec
	e^{\deter}(n,\Int,F_{\mon}^d) \asymp n^{-1/d} \,,
\end{multline*}
where the hidden constants depend on~$d$.
It is an open problem to find lower bounds for the non-adaptive Monte Carlo error
that actually show that adaption is better for~\mbox{$d \geq 2$} as well,
but from the one-dimensional case we conjecture it to be like that.

For the $L_1$-approximation, however,
the order of convergence does not reveal any differences
between the various algorithmic settings.
Applying Papageorgiou's proof technique to the problem of $L_1$-approximation,
we obtain the following theorem.
\begin{theorem}\label{thm:MonAppOrderConv}
	For the $L_1$-approximation of monotone functions, for fixed dimension~$d$
	and \mbox{$n\rightarrow \infty$}, we have the following asymptotic behaviour,
	\begin{equation*}
		e^{\ran}(n,\App,F_{\mon}^d)
			\asymp e^{\deter}(n,\App,F_{\mon}^d)
			\asymp n^{-1/d} \,,
	\end{equation*}
	where the implicit constants depend on~$d$.
\end{theorem}
\begin{proof}
	We split~\mbox{$[0,1]^d$} into $m^d$~subcubes indexed
	by~\mbox{$\veci \in \{0,1,\ldots,m-1\}^d$}:
	\begin{equation*} 
		C_{\veci} := \bigtimes_{j = 1}^{d} I_{i_j}
	\end{equation*}
	where~\mbox{$I_i := [{\textstyle \frac{i}{m}}, {\textstyle \frac{i+1}{m}})$}
	for~\mbox{$i=0,1,\ldots,m-2$}
	and~\mbox{$I_{m-1} := [\textstyle{\frac{m-1}{m}},1]$}.
	
	For the lower bounds, we consider fooling functions~$f = f_{\vecdelta}$
	that are constant on each of the subcubes, in detail,
	\begin{equation*}
		f|_{C_{\veci}} \,=\, \frac{2(|\veci|_1+\delta_{\veci})}{d(m-1)+1} - 1
	\end{equation*}
	with $\delta_{\veci} \in \{0,1\}$
	and \mbox{$|\veci|_1 := i_1 + \ldots + i_d$}.
	Obviously, such functions are monotonically increasing.
	In order to obtain lower bounds that hold for Monte Carlo algorithms,
	we employ a minimax argument,
	also known as \emph{Bakhvalov's trick}~\cite{Bakh59}.
	Namely, we average over all possible settings
	of~$\vecdelta = (\delta_{\veci})$.
	For any information~$\vecy$,
	let \mbox{$I^{\vecy} \subset \{0,\ldots,m-1\}^d$}
	be the set of indices~$\veci$
	where we do not know anything about the function
	on the corresponding subcube~$C_{\veci}$.
	For an algorithm which uses $n < m^d$ function values,
	we have \mbox{$\# I^{\vecy} \geq m^d-n$}.
	Considering an arbitrary Monte Carlo algorithm
	$A_n^{\omega} = \phi^{\omega} \circ N^{\omega}$
	we can write
	\begin{multline*}
		e((A_n^{\omega}),F_{\mon}^d)
			\,=\, \sup_{f \in F_{\mon}^d} \expect \|A_n^{\omega}(f) - f\|_{L_1}
			\,\geq\, 2^{-m^d} \sum_{(\delta_{\veci}) \in \{0,1\}^{m^d}}
						\expect \|A_n^{\omega}(f_{\vecdelta}) - f_{\vecdelta}\|_{L_1} \\
			\,\geq\, \expect 2^{-m^d} \sum_{\vecdelta = (\delta_{\veci})}
						\sum_{\substack{\veci \in I^{\vecy}\\
														\text{where }\vecy := N^{\omega}(f_{\vecdelta})\\
														\text{and }g := \phi^{\omega}(\vecy)}}
					\int_{C_{\veci}}
						\underbrace{\fracts{1}{2}
													\left(\Bigl|\fracts{2|\veci|_1}{d(m-1)+1}
																			- g(\vecx)
																\Bigr|
															 +\Bigl|\fracts{2(|\veci|_1 + 1)}{d(m-1)+1}
																			- g(\vecx)
																\Bigr|
													\right)
												}_{\textstyle \geq \frac{1}{d\,(m-1)+1}}
							\dint \vecx\\
			\,\geq\, \left(1 - \frac{n}{m^d}\right) \, \frac{1}{d(m-1)+1} \,.
	\end{multline*}
	Choosing~\mbox{$m := \lceil (2n)^{1/d} \rceil$}, 
	we obtain the general lower bound
	\begin{equation*}
		e^{\deter}(n,\App,F_{\mon}^d)
			\,\geq\, e^{\ran}(n,\App,F_{\mon}^d)
			\,\geq\, \frac{1}{2 \, (d \cdot (2n)^{1/d} + 1)}
			\,\geq\, \frac{1}{6 \, d} \, n^{-1/d} \,.
	\end{equation*}
	
	For the upper bounds, we give a deterministic, non-adaptive algorithm
	with cardinality~\mbox{$(m-1)^d$},
	i.e.~when allowed to use $n$~function values,
	we choose~\mbox{$m := \lfloor n^{1/d} \rfloor + 1$}.
	Splitting the domain into $m^d$ subcubes as above,
	we compute \mbox{$(m-1)^d$} function values
	at the corner points in the interior~\mbox{$(0,1)^d$} of the domain.
	For each subcube we take the medium possible value based on our knowledge
	on the function~$f$ in the lower and upper corners
	of that particular subcube,
	where at the boundary without computing function values we assume
	\begin{equation*}
		f|_{[0,1)^d \setminus (0,1)^d} = -1
			\quad\text{and}\quad
		f|_{[0,1]^d \setminus [0,1)^d} = 1 \,.
	\end{equation*}
	The subcubes can be grouped into diagonals,
	where the upper corner of one subcube touches the lower corner of the next
	subcube.
	Each diagonal can be uniquely represented
	by an index~$\veci$ with at least one $0$-entry,
	which thus belongs to the lowest subcube~$C_{\veci}$ of that diagonal,
	in total we have \mbox{$m^d - (m-1)^d \leq d \, m^{d-1}$} diagonals.
	Due to monotonicity, the contribution of a single diagonal
	to the~$L_1$-error is at most~$m^{-d}$,
	so altogether we have
	\begin{equation*}
		e(A_m^d,f)
			\,\leq\, \frac{d \, m^{d-1}}{m^d}
			\,=\, \frac{d}{m}
			\,=\, \frac{d}{\lfloor n^{1/d} \rfloor + 1}
			\,\leq\, d \, n^{-1/d} \,.
	\end{equation*}
	For details compare also the proofs for integration
	in Papageorgiou~\cite{Papa93}.
\end{proof}

The Monte Carlo lower bound contained in the above result
does also hold for algorithms with varying cardinality,
see~\cite[Theorem~4.2]{Ku17}.

\begin{remark}[On the impracticality of these results]
	\label{rem:MonAppOrderConv}
	The above proof yields the explicit estimate
	\begin{equation*}
		\frac{1}{6 \, d} \, n^{-1/d}
			\,\leq\, e^{\ran}(n,\App,F_{\mon}^d)
			\,\leq\, e^{\deter}(n,\App,F_{\mon}^d)
			\,\leq\, d \, n^{-1/d} \,.
	\end{equation*}
	At first glance, this estimate appears friendly, with constants
	differing only polynomially in~$d$.
	This optimistic view, however, collapses dramatically
	when switching to the notion of $\eps$-complexity
	for~\mbox{$0<\eps<1$}:
	\begin{equation*}
		\left(\frac{1}{6 \, d}\right)^d \, \eps^{-d}
			\,\leq\, n^{\ran}(\eps,\App,F_{\mon}^d)
			\,\leq\, n^{\deter}(\eps,\App,F_{\mon}^d)
			\,\leq\, d^d \, \eps^{-d} \,.
	\end{equation*}
	Here, the constants differ superexponentially in~$d$.
	Of course, lower bounds for low dimensions also hold for higher
	dimensions,
	so given the dimension~$d_0$, one can optimize over~\mbox{$d=1,\ldots,d_0$}.
	Still, the upper bound is impractical for high dimensions since it
	is based on algorithms that use exponentially~(in~$d$) many function values.
	
	In fact, for the deterministic setting we cannot avoid a bad~$d$-dependency,
	we suffer from the curse of dimensionality, see \secref{sec:Curse}.
	For the randomized setting, however,
	we can significantly reduce the $d$-dependency (which is still high),
	at least as long as~$\eps$ is fixed,
	see \secref{sec:UBs}.
	To summarize, if we only consider the order of convergence,
	we might think that randomization does not help,
	but for high dimensions randomization actually \emph{does} help,
	at least in the preasymptotic regime.
\end{remark}

\subsection{Curse of dimensionality in the deterministic setting}
\label{sec:Curse}

Hinrichs, Novak, and Wo\'{z}niakowski~\cite{HNW11} have shown
that the integration
(and hence also the $L_p$-approximation, \mbox{$1 \leq p \leq \infty$})
of monotone functions suffers from the curse of dimensionality
in the deterministic setting.

\begin{theorem}[Hinrichs, Novak, Wo\'{z}niakowski 2011] \label{thm:HNW11}
	The $L_1$-approximation of monotone functions
	suffers from the curse of dimensionality in the worst case setting.
	In detail,
	\begin{equation*}
		e^{\deter}(n,F_{\mon}^d)
			\,\geq\, \left(1 - n \, 2^{-d}\right) \,,
	\end{equation*}
	so for~\mbox{$0 < \eps \leq 1/2$}
	we have
	\begin{equation*}
		n^{\deter}(\eps,F_{\mon}^d)
			\,\geq\, 2^{d-1} \,.
	\end{equation*}
\end{theorem}
\begin{proof}[Idea of the proof]
	Any deterministic algorithm will fail
	to distinguish the diagonal split function
	$f_{\boxbslash}(\vecx) := \sgn\left(\sum_{j=1}^d x_j - \frac{d}{2}\right)$
	from other monotone functions~$F_{\mon}^d$
	which yield the same information.
	No matter what information mapping~$N$ we take,
	there will exist such indistinguishable functions with a big~$L_1$-distance
	to $f_{\boxbslash}$,
	since in this situation each function value only provides knowledge
	about a subdomain of volume at most~$2^{-d}$,
	see~\cite{HNW11} for details.
\end{proof}

Note that the \emph{initial error}~$e(0,F_{\mon}^d)$ is~$1$,
this means, if we do not know any function value,
the best guess is the zero function.
Thus the theorem above states that in order to merely halve the initial error
we already need exponentially (in~$d$) many function values.
The curse of dimensionality can be broken via a Monte Carlo method,
see \secref{sec:UBs},
but we still have intractability in the randomized setting,
see \secref{sec:LBs}.
In contrast, for \emph{integration} the standard Monte Carlo method
\begin{equation*}
	M_n(f) \,:=\, \frac{1}{n} \sum_{i=1}^n f(\vecX_i)
		\,\approx\, 
						\Int(f) \,,
	\qquad \text{where $\vecX_i \stackrel{\text{iid}}{\sim} \Uniform([0,1]^d)$,}
\end{equation*}
easily achieves \emph{strong polynomial tractability},
namely~$n(\eps,\Int,F_{\mon}^d) \leq \lceil\eps^{-2}\rceil$,
where the dimension~$d$ does not play any role.

\section{Breaking the curse with Monte Carlo}
\label{sec:UBs}

We present and analyse a new algorithm for the approximation of
monotone functions on the unit cube.
It is the first algorithm to show that for this problem
the curse of dimensionality does not hold in the randomized setting.
The idea of the algorithm has been inspired by a method
for learning Boolean monotone functions due to Bshouty and Tamon~\cite{BT96}.

The method is based on the Haar wavelet decomposition
of the function~$f$.
We define dyadic cuboids on~\mbox{$[0,1]^d$}
indexed by~\mbox{$\vecalpha \in \N^d$}, 
or equivalently by an index vector pair~\mbox{$(\veclambda,\veckappa)$}
with~\mbox{$\veclambda \in \N_0^d$}
and~\mbox{$\veckappa \in \N_0^d$},
\mbox{$\kappa_j < 2^{\lambda_j}$},
such that~\mbox{$\alpha_j = 2^{\lambda_j} + \kappa_j$}
for~\mbox{$j = 1,\ldots,d$}:
\begin{equation*}
	C_{\vecalpha} = C_{\veclambda,\veckappa}
		:= \bigtimes_{j=1}^d I_{\alpha_j} \,,
\end{equation*}
where
\begin{equation*}
	I_{\alpha_j} = I_{\lambda_j,\kappa_j}
		:=\begin{cases}
				[\kappa_j \, 2^{-\lambda_j} , (\kappa_j+1) \, 2^{-\lambda_j})
					\quad&\text{for~$\kappa_j = 0,\ldots,2^{\lambda_j}-2$,} \\
				[1 - 2^{-\lambda_j} , 1]
					\quad&\text{for~$\kappa_j = 2^{\lambda_j}-1$.}
			\end{cases}
\end{equation*}
Note that for fixed~$\lambda_j$ we have a decomposition
of the unit interval~\mbox{$[0,1]$}
into $2^{\lambda_j}$~disjoint intervals of length~$2^{-\lambda_j}$.
(This index system for subdomains
differs from the index system for subcubes in \secref{sec:Order}, 
which shall be no source of confusion.)
One-dimensional Haar wavelets~\mbox{$\psi_{\alpha_j} : [0,1] \rightarrow \R$}
are defined for~\mbox{$\alpha_j \in \N_0$}
(if~\mbox{$\alpha_j=0$}, we put~\mbox{$\lambda_j = -\infty$} and~\mbox{$\kappa_j=0$}),
\begin{equation*}
	\psi_{\alpha_j} 
		:=\begin{cases}
				\ind_{[0,1]}
					\quad&\text{if $\alpha_j = 0$
											(i.e.\ $\lambda_j = -\infty$ and $\kappa_j=0$),}\\
				2^{\lambda_j/2}
					\, (\ind_{I_{\lambda_j+1, 2\kappa_j+1}}
							- \ind_{I_{\lambda_j+1, 2\kappa_j}})
					\quad&\text{if $\alpha_j \geq 1$ (i.e.\ $\lambda_j \geq 0$).}
			\end{cases}
\end{equation*}
In~\mbox{$L_2([0,1]^d)$} we have the orthonormal basis
\mbox{$\{\psi_{\vecalpha}\}_{\vecalpha \in \N_0^d}$} with
\begin{equation*}
	\psi_{\vecalpha}(\vecx) := \prod_{j=1}^d \psi_{\alpha_j}(x_j) \,.
\end{equation*}
The volume of the support of~$\psi_{\vecalpha}$
is~\mbox{$2^{-|\veclambda|_{+}}$}
with~\mbox{$|\veclambda|_{+} := \sum_{j=1}^d \max\{0,\lambda_j\}$}.
The basis function~$\psi_{\vecalpha}$ only takes discrete
values~\mbox{$\{0,\pm 2^{|\veclambda|_{+}/2}\}$},
hence it is normalized indeed.
(Our definition differs from the usual definition of the Haar basis
where for each~$\alpha_j > 0$ the sign would be reversed.
However, our version is convenient in the context of monotone functions,
especially in the proof of \lemref{lem:monSmallWavelet}.)

We can write any monotone function~$f$ as the Haar wavelet decomposition
\begin{equation*}
	f \,=\, \sum_{\vecalpha \in \N_0^d} \tilde{f}(\vecalpha) \, \psi_{\vecalpha}
\end{equation*}
with the wavelet coefficients
\begin{equation*}
	\tilde{f}(\vecalpha)
		\,:=\, \langle \psi_{\vecalpha}, f \rangle
		\,=\, \expect \psi_{\vecalpha}(\vecX) \, f(\vecX) \,,
\end{equation*}
where~$\vecX$ is uniformly distributed on~\mbox{$[0,1]^d$}.
The algorithm shall use information from random samples
\mbox{$f(\vecX_1),\ldots,f(\vecX_n)$},
with~\mbox{$\vecX_i \stackrel{\text{iid}}{\sim} \Uniform [0,1]^d$},
in order to approximate the most important wavelet coefficients
via the standard Monte Carlo method,
\begin{equation} \label{eq:g(alpha)}
	\tilde{f}(\vecalpha)
		\,\approx\, \tilde{h}(\vecalpha)
		\,:=\, \frac{1}{n} \sum_{i=1}^n \psi_{\vecalpha}(\vecX_i) \, f(\vecX_i) \,.
\end{equation}
In particular, we choose a resolution~\mbox{$r \in \N$},
and a parameter~\mbox{$k \in \{1,\ldots,d\}$},
and only consider indices~\mbox{$\vecalpha \leftrightarrow (\veclambda,\veckappa)$}
with~\mbox{$\lambda_j < r$}
and~\mbox{$|\vecalpha|_0 := \#\{ j \mid \alpha_j > 0\} \leq k$}.
(The quantity $|\vecalpha|_0$ counts the number of \emph{active variables}
of a wavelet~$\psi_{\vecalpha}$.)
A naive linear algorithm would simply return a linear reconstruction,
\begin{equation} \label{eq:linAlg}
	h \,=\, A_{n,k,r}^{\omega}(f)
		\,:=\, \sum_{\substack{\vecalpha \in \N_0^d\\
						|\vecalpha|_0 \leq k \\
						\veclambda < r}}
				\tilde{h}(\vecalpha) \, \psi_{\vecalpha}
			 \,, \qquad \text{for $f \in F_{\mon}^d$\,.}
\end{equation}
This linear algorithm can already break the curse of dimensionality
but the~$\eps$-dependency of the required sample size is unfavourable,
see~\cite[Theorem~4.22]{Ku17} for a detailed analysis.
Instead, for the subclass of sign-valued monotone functions
\begin{equation*}
	F_{\mon\pm}^d
		:= \{f:[0,1]^d \rightarrow \{-1,+1\} \mid f \in F_{\mon}^d\} \,,
\end{equation*}
in the~$L_1$-approximation setting
it is natural to return a sign-valued approximation,
\begin{equation} \label{eq:AlgF+-}
	g \,=\, \hat{A}_{n,k,r}^{\omega}(f)
		\,:=\, \sgn h
		\,=\, \sgn \circ\,[A_{n,k,r}^{\omega}(f)]
	\,, \qquad \text{for $f \in F_{\mon\pm}^d$\,.}
\end{equation}
(Here and for the rest of this paper,
we put~$\sgn(0) := 1$ in order to avoid zero values.)
For general monotone functions \mbox{$f \in F_{\mon}^d$}
with function values in~$[-1,+1]$,
the algorithm can be generalized to
\begin{equation} \label{eq:AlgF}
	\bar{A}_{n,k,r}^{\omega}(f)
		\,:=\, \frac{1}{2} \int_{-1}^1
					\hat{A}_{n,k,r}^{\omega} (f_t)
						\dint t \,,
	\qquad \text{where $f_t(\vecx) := \sgn(f(\vecx)-t)$ for~$t \in \R$.}
\end{equation}
Note that the function values~$f_t(\vecX_i)$ which are needed in the course
of evaluating~$\bar{A}_{n,k,r}$ can be directly derived from
function values~$f(\vecX_i)$, so we still use the same information
as within the simple linear algorithm~\eqref{eq:linAlg}.
(This trick would not be possible for algorithms
with an adaptive procedure for collecting information
on a sign-valued function.)
The idea for the generalized algorithm~$\bar{A}_{n,k,r}$
is based on the observation
\begin{equation} \label{eq:observation}
	f(\vecx) \,=\, \frac{1}{2} \int_{-1}^1 \sgn(f(\vecx)-t) \dint t \,,
		\qquad \text{for $f(\vecx) \in [-1,1]$.}
\end{equation}
The validity of this approach is summarized in the following Lemma.
\begin{lemma} \label{lem:FvsF+-}
	For the approximation of monotone functions with the methods
	defined in~\eqref{eq:AlgF+-} and \eqref{eq:AlgF} we have
	\begin{equation*}
		e(\hat{A}_{n,k,r},F_{\mon\pm}^d)
			\,=\, e(\bar{A}_{n,k,r},F_{\mon}^d) \,.
	\end{equation*}
\end{lemma}
\begin{proof}
	Since for sign-valued functions $f \in F_{\mon\pm}^d$
	we have $f_t = f$ for~$t \in (-1,1]$,
	trivially~$\hat{A}_{n,k,r}(f) = \bar{A}_{n,k,r}(f)$,
	and from~$F_{\mon\pm}^d \subset F_{\mon}^d$
	we conclude the inequality
	\mbox{$e(\hat{A}_{n,k,r},F_{\mon\pm}^d)
					\leq e(\bar{A}_{n,k,r},F_{\mon}^d)$}.
	
	Reversely,
	for~$f \in F_{\mon}^d$,
	using the definition of~$\bar{A}_{n,k,r}$ in \eqref{eq:AlgF}
	and the observation~\eqref{eq:observation},
	via the triangle-inequality and Fubini's theorem we have
	\begin{align*}
		e(\bar{A}_{n,k,r},f)
			&\,=\, \expect \left\|f - \bar{A}_{n,k,r}^{\omega}(f)\right\|_{L_1} \\
		[\eqref{eq:AlgF},\eqref{eq:observation}]\qquad
			&\,=\,
				\frac{1}{2} \expect \left\| \int_{-1}^{1}
																(f_t - \hat{A}_{n,k,r}^{\omega}(f_t))
															\dint t
										\right\|_{L_1} \\
		[\text{$\Delta$-ineq., Fubini}]\qquad
			&\,\leq\,
					\frac{1}{2}
						\int_{-1}^1
								\expect \bigl\|f_t - \hat{A}_{n,k,r}^{\omega}(f_t)\bigr\|
							\dint t
			\,\leq\, e(\hat{A}_{n,k,r},F_{\mon\pm}^d) \,.
	\end{align*}
	This implies
	$e(\bar{A}_{n,k,r},F_{\mon}^d) \leq e(\hat{A}_{n,k,r},F_{\mon\pm}^d)$,
	thus finishing the proof.
\end{proof}

We continue with the error analysis of the given algorithm,
where by virtue of the above lemma
we may restrict to the approximation of sign-valued monotone functions
via~$\hat{A}_{n,k,r}$.
For details on the implementation of~$\bar{A}_{n,k,r}$,
see \remref{rem:monoMCUBphicost}.

A key result for the error analysis is the following fact
about those Haar wavelet coefficients which are dropped by the algorithm,
compare Bshouty and Tamon~\cite[Section~4]{BT96} for the Boolean setting.
\begin{lemma}\label{lem:monSmallWavelet}
	For any monotone function~$f \in F_{\mon}^d$ we have
	\begin{equation*}
		\sum_{\substack{\vecalpha \in \N_0^d\\
											|\vecalpha|_0 > k \\
											\veclambda < r}}
				\tilde{f}(\vecalpha)^2
			\,\leq\, \frac{\sqrt{d \, r}}{k+1} \,.
	\end{equation*}
\end{lemma}
\begin{proof}
	Within the first step, we consider special wavelet
	coefficients~\mbox{$\tilde{f}(\alpha \, \vece_j)$}
	that measure the average growth of~$f$
	along the $j$-th coordinate within
	the interval~\mbox{$I_{\alpha}$},
	where~$\vece_j$ is the $j$-th vector of the standard basis in~$\R^d$.
	We will frequently use the alternative
	indexing~\mbox{$I_{\lambda,\kappa}$}
	with~\mbox{$\alpha = 2^{\lambda} + \kappa \in \N$},
	where~\mbox{$\lambda \in \N_0$} and~\mbox{$\kappa = 0,\ldots,2^{\lambda}-1$}.
	We define the function
	\begin{align*}
		f_{\alpha j}(\vecx)
			&:= \begin{cases}
						0
							\quad&\text{for $x_j \notin I_{\alpha}$,}\\
						2^{\lambda/2} \,
							\int_0^1 \psi_{\alpha}(z_j) \cdot 
										f(\vecz)\Big|_{\substack{z_{j'} = x_{j'}\\
																						\text{for $j' \not= j$}}}
								\dint z_j
							\quad&\text{for $x_j \in I_{\alpha}$.}
					\end{cases}
	\end{align*}
	Due to the monotonicity of~$f$
	we have \mbox{$f_{\alpha j} \geq 0$},
	and from the boundedness of~$f$ we have \mbox{$f_{\alpha j} \leq 1$}.
	Using this and Parseval's equation, we obtain
	\begin{align*}
		\tilde{f}(\alpha \, \vece_j)
			\,=\, \langle \psi_{\alpha \, \vece_j} , f \rangle
			\,=\, 2^{\lambda/2} \, \left\| f_{\alpha j} \right\|_{L_1}
			\,\geq\, 2^{\lambda/2} \,
					\left\| f_{\alpha j} \right\|_{L_2}^2
			\,=\, 2^{\lambda/2} \, \sum_{\vecalpha' \in \N_0^d}
						\left\langle \psi_{\vecalpha'},
													f_{\alpha j}
						\right\rangle^2 \,.
	\end{align*}
	Since the function~$f_{\alpha j}$
	is constant in~$x_j$ on~$I_{\alpha}$ and vanishes outside,
	we only need to consider summands
	with coarser resolution~\mbox{$\lambda_j' < \lambda$}
	in that coordinate,
	and where the support of~$\psi_{\vecalpha}$
	contains the support of~$f_{\alpha j}$.
	That is the case for
	\mbox{$\kappa_j' = \lfloor 2^{\lambda_j' - \lambda} \kappa \rfloor$}
	with~\mbox{$\lambda_j' = -\infty,0,\ldots,\lambda-1$}.
	For such indices $\vecalpha' \leftrightarrow (\veckappa',\veclambda')$
	we have
	\begin{equation*}
		\left\langle \psi_{\vecalpha'}, f_{\alpha j} \right\rangle^2
			\,=\, 2^{\max\{0,\lambda_j'\} - \lambda} \,
							\left\langle \psi_{\vecalpha''},f
							\right\rangle^2
			\,=\, 2^{\max\{0,\lambda_j'\} - \lambda} \, \tilde{f}^2(\vecalpha'') \,,
	\end{equation*}
	where~\mbox{$\alpha''_{j'} = \alpha'_{j'}$} for~\mbox{$j' \not= j$},
	and~\mbox{$\alpha''_j = \alpha$}. Hence we obtain
	\begin{equation} \label{eq:f(aej)>}
		\tilde{f}(\alpha \, \vece_j)
			\,\geq\, 2^{\lambda/2}
					\, \underbrace{\left(2^{-\lambda}
															+ \sum_{l = 0}^{\lambda-1} 2^{l-\lambda}
													\right)
												}_{= 1}
					\, \sum_{\substack{\vecalpha'' \in \N_0^d\\
														\alpha_j'' = \alpha}}
								\tilde{f}^2(\vecalpha'') \,.
	\end{equation}
	
	Based on this relation between the wavelet coefficients, we can estimate
	\begin{align*}
		1 \,=\, \|f\|_{L_2}^2
			&\,=\,
				\sum_{\vecalpha \in \N_0^d} \tilde{f}^2(\vecalpha)
			\,\geq\,
				\sum_{j=1}^d
					\sum_{\lambda = 0}^{r-1}
						\sum_{\kappa = 0}^{2^{\lambda}-1}
							\tilde{f}^2((2^{\lambda}+\kappa) \, \vece_j)\\
			&\,\geq\,
				\sum_{j=1}^d \sum_{\lambda = 0}^{r-1} 2^{-\lambda}
					\Biggl(\sum_{\kappa = 0}^{2^{\lambda}-1}
											\tilde{f}((2^{\lambda}+\kappa) \, \vece_j)
					\Biggr)^2\\
		[\eqref{eq:f(aej)>}]\qquad
			&\,\geq\,
				\sum_{j=1}^d \sum_{\lambda = 0}^{r-1}
					\Biggl(\sum_{\substack{\vecalpha \in \N_0^d\\
														\lambda_j = \lambda}}
									\tilde{f}^2(\vecalpha)
					\Biggr)^2 \,.
	\end{align*}
	Taking the square root, and using the norm estimate
	\mbox{$\|\vecv\|_1 \leq \sqrt{m} \, \|\vecv\|_2$}
	for~\mbox{$\vecv \in \R^m$},
	here with~\mbox{$m = dr$}, we end up with
	\begin{equation*}
		1 \,\geq\, \frac{1}{\sqrt{d r}} \,
					\sum_{j=1}^d \sum_{\lambda = 0}^{r-1}
									\sum_{\substack{\vecalpha \in \N_0^d\\
														\lambda_j = \lambda}}
										\tilde{f}^2(\vecalpha)
			\,\geq\, \frac{1}{\sqrt{d r}} \,
					\sum_{\substack{\vecalpha \in \N_0^d\\
														\veclambda < r}}
						|\vecalpha|_0 \, \tilde{f}^2(\vecalpha)
			\,\geq\, \frac{k+1}{\sqrt{d r}} \,
					\sum_{\substack{\vecalpha \in \N_0^d\\
														|\vecalpha|_0 > k \\
														\veclambda < r \\}}
										\tilde{f}^2(\vecalpha) \,.
	\end{equation*}
	This proves the lemma.
\end{proof}

By virtue of the above lemma
we obtain the following error and complexity bound,
compare Bshouty and Tamon~\cite[Theorem~5.1]{BT96}
for the Boolean setting.
(Their error criterion is slightly different,
namely, it is the $L_1$-error at prescribed confidence level
rather than the expected $L_1$-distance.)

\begin{theorem} \label{thm:monoUBs}
	For the algorithm~\mbox{$\bar{A}_{n,k,r}
		= (\bar{A}_{n,k,r}^{\omega})_{\omega \in \Omega}$}
	as defined in \eqref{eq:AlgF}
	we have the error bound
	\begin{equation*}
		e(\bar{A}_{n,k,r},F_{\mon}^d)
			\,\leq\, 5 \, \frac{d}{2^{r}}
				+ 4\,\frac{\sqrt{d r}}{k+1}
								+ 4\,\frac{\exp[k(1+\log \frac{d}{k} + (\log 2) \, r)]}{n}
								 \,.
	\end{equation*}
	Given~\mbox{$0 < \eps < 1$},
	the $\eps$-complexity for the Monte Carlo approximation
	of monotone functions is bounded by
	\begin{align*}
		n^{\ran}(\eps,F_{\mon}^d)
			\,\leq\, \min\Biggl\{&
								\exp\left[C \, 
																	\frac{\sqrt{d}}{\eps} 
														\, \left( 1 + \log \frac{d}{\eps}\right)^{3/2}
													\right],\,
								\exp\left[d \, \log \frac{d}{2 \eps}\right]
							\Biggr\} \,,
	\end{align*}
	with some numerical constant~\mbox{$C > 0$}.
	In particular, the curse of dimensionality does \emph{not} hold
	for the randomized $L_1$-approximation of monotone functions.
\end{theorem}
\begin{proof}
	Thanks to \lemref{lem:FvsF+-}, we may restrict
	to the analysis of the algorithm~$\hat{A}_{n,k,r}$ for sign-valued
	functions~$f \in F_{\mon\pm}^d$.
	
	Since we only take certain wavelet coefficients
	until a resolution~$r$ into account,
	the reconstruction~\eqref{eq:AlgF+-}
	will be a function which is constant on each
	of~$2^{rd}$ subcubes~$C_{r \ones,\veckappa}$
	where~\mbox{$\ones = (1,\ldots,1)$}
	and \mbox{$\veckappa \in \{0,\ldots,2^r-1\}^d$}.
	The algorithm can be seen as actually approximating
	\begin{equation*}
		\sgn f_r \,,
		\qquad \text{where}\quad
		f_r := \sum_{\substack{\vecalpha \in \N_0^d\\
																			\veclambda < r}}
												\tilde{f}(\vecalpha) \, \psi_{\vecalpha} \,.
	\end{equation*}
	Since on the one hand,
	the Haar wavelets are constant on each of these~$2^{rd}$~subcubes,
	and on the other hand,
	we have $2^{rd}$~wavelets up to this resolution,
	the function~$f_r$ takes on each of the subcubes
	the average function value of~$f$ on that subcube,
	which is between~$-1$ and $+1$.
	The function $\sgn f_r$ takes the subcubewise predominant value of~$f$,
	which is either $-1$ or~$+1$.
	That is,
	for~\mbox{$\vecX,\vecX' \sim \Uniform C_{r \ones,\veckappa}$}
	we have
	\begin{equation} \label{eq:frsubcube}
		\expect|f(\vecX) - \sgn f_r(\vecX)|
			\,=\, \expect|f(\vecX) - \median' f(\vecX')|
			\,\leq\, \ind[\text{$f$ not const.\ on~$C_{r\ones,\veckappa}$}] \,,
	\end{equation}
	and for~$\vecx \in C_{r\ones,\veckappa}$ we can estimate
	\begin{equation} \label{eq:sgnfrsubcube}
		|\sgn(f_r(\vecx)) - f_r(\vecx)|
			\,\leq\, \ind[\text{$f$ not const.\ on~$C_{r\ones,\veckappa}$}] \,.
	\end{equation}
	Similarly to the upper bound part
	within the proof of \thmref{thm:MonAppOrderConv},
	we group the subcubes into diagonals.
	By monotonicity,
	there is at most one subcube within each diagonal
	where the sign-valued function~$f$ jumps from~$-1$ to~$+1$,
	hence~\eqref{eq:frsubcube} and~\eqref{eq:sgnfrsubcube} are non-zero
	but bounded by~$1$.
	Now that there are
	\mbox{$2^{rd} - (2^r-1)^d
					\leq d \cdot 2^{r (d-1)}$}~%
	diagonals,
	and the volume of each subcube is~$2^{-rd}$,
	we obtain
	\begin{equation} \label{eq:|f-fr|}
			\|f - \sgn f_r\|_{L_1} \,\leq\, \frac{d \cdot 2^{r(d-1)}}{2^{rd}}
												\,=\, \frac{d}{2^r} \,,
			\qquad\text{as well as}\qquad
			\|\sgn f_r - f_r\|_{L_2}^2 \,\leq\, \frac{d}{2^r} \,.
	\end{equation}
	
	Surprisingly, the fact that the wavelet basis functions have a small
	support, actually helps to keep the error for estimating the wavelet
	coefficients small. Exploiting independence of the sample points
	and unbiasedness of the standard Monte Carlo wavelet coefficient 
	estimator~\eqref{eq:g(alpha)},
	for~\mbox{$(\veclambda,\veckappa)
								\leftrightarrow \vecalpha \in \N_0^d$}
	we have
	\begin{align}
		\expect [\tilde{h}(\vecalpha)-\tilde{f}(\vecalpha)]^2
			&\,=\, \expect \left[\frac{1}{n}
												\sum_{i=1}^n \psi_{\vecalpha}(\vecX_i) \, f(\vecX_i)
																			- \tilde{f}(\vecalpha)
								\right]^2
					\nonumber\\
		\text{[$\vecX_i$ i.i.d.; unbiasedness]}\qquad
			&\,=\, \frac{1}{n} \expect\left[\psi_{\vecalpha}(\vecX_1) \, f(\vecX_1)
																	- \tilde{f}(\vecalpha)\right]^2 
					\nonumber\\
			&\,\leq\, \frac{1}{n}
								\expect \left[\psi_{\vecalpha}(\vecX_1) \, f(\vecX_1)\right]^2
					\nonumber\\
			&\,=\, \frac{1}{n}
						\, \underbrace{\P\{ \vecX_1 \in C_{\vecalpha} \}
													}_{= 2^{-|\veclambda|_+}}
						\, \expect\bigl[\underbrace{(\psi_{\vecalpha}(\vecX_1)
																					\, f(\vecX_1))^2
																				}_{= 2^{|\veclambda|_+}}
												\mid \vecX_1 \in C_{\vecalpha} 
											\bigr]
					\nonumber\\
			&\,=\, \frac{1}{n} \,.
					\label{eq:waveletEst}
	\end{align}
	
	This estimate on the quality of wavelet coefficient approximation
	can be used for estimating $L_2$-errors.
	Regarding the approximation~$g = \hat{A}^{\omega}_{n,k,r}(f) = \sgn h$
	as defined in~\eqref{eq:AlgF+-},
	from the observation
	\begin{equation*}
		\sgn (f_r(\vecx)) \,\not=\, g(\vecx) \,=\, \sgn(h(\vecx))
		\quad\Rightarrow\quad
		(\sgn (f_r(\vecx)) - h(\vecx))^2 \,\geq\, 1
	\end{equation*}
	we conclude
	\begin{equation} \label{eq:L_1<->L_2}
			\|\sgn f_r - g \|_{L_1}
				\,\leq\, 2\|\sgn f_r - h\|_{L_2}^2
				\,\leq\, 4\, \|\sgn f_r - f_r\|_{L_2}^2 + 4\,\|f_r - h\|_{L_2}^2 \,.
	\end{equation}
	
	Then, combining previous estimates,
	the expected distance between the input~$f$
	and the approximate reconstruction~$g = \hat{A}^{\omega}_{n,k,r}(f) = \sgn h$
	from~\eqref{eq:AlgF+-}
	can be bounded as follows,
	\begin{align}
		\expect \|f - g\|_{L_1}
			&\,\leq\, \|f - \sgn f_r\|_{L_1} + \expect\|\sgn f_r - g\|_{L_1}
					\nonumber\\
		[\eqref{eq:L_1<->L_2}]\qquad
			&\,\leq\,
					\|f - \sgn f_r\|_{L_1}
						+ 4\, \|\sgn f_r - f_r\|_{L_2}^2 + 4 \expect\|f_r - h\|_{L_2}^2
					\nonumber\\
		[\text{\eqref{eq:|f-fr|}, Parseval}]\qquad
			&\,\leq\,
				5\,\frac{d}{2^r}
					+ 4\,\biggl(\sum_{\substack{\vecalpha \in \N_0^d\\
																	|\vecalpha|_0 > k \\
																	\veclambda < r}}
										\tilde{f}(\vecalpha)^2
									+ \sum_{\substack{\vecalpha \in \N_0^d\\
																		|\vecalpha|_0 \leq k \\
																		\veclambda < r}}
											\expect [\tilde{f}(\vecalpha)
																- \tilde{h}(\vecalpha)]^2
						\biggr)
				\nonumber\\
		\text{[\lemref{lem:monSmallWavelet}; \eqref{eq:waveletEst}]}\qquad
			&\,\leq\, 5\,\frac{d}{2^r}
									+ 4\, \frac{\sqrt{d \, r}}{k+1}
									+ 4\,\frac{\# A}{n}\,,
					\label{eq:error3summands}
	\end{align}
	where~$A$ is the index set corresponding to the wavelet coefficients
	that are computed,
	\begin{equation*}
		A := \{\vecalpha \in \N_0^d
						\mid |\vecalpha|_0 \leq k
								\text{ and }
								\veclambda < r\} \,.
	\end{equation*}
	We can quantify the size of the index set~$A$
	for~\mbox{$k \in \{1,\ldots,d\}$}
	by standard estimates,
	\begin{equation*}
		\# A \,=\, \sum_{l=0}^k \binom{d}{l} \, (2^r - 1)^l
			\,\leq\, 2^{r k} \sum_{l=0}^k \binom{d}{l}
			\,\leq\, 2^{r k} \, \left(\frac{\euler \, d}{k}\right)^k \,.
	\end{equation*}
	This finally yields the error bound
	for the Monte Carlo method~$\hat{A}_{n,k,r}$
	applied to sign-valued functions~$f \in F_{\mon\pm}^d$
	as stated in the theorem.
	
	Choosing the resolution~\mbox{$r := \lceil \log_2 \frac{15\,d}{\eps}
																			\rceil$}
	will bound the first term~\mbox{$5 \, d \cdot 2^{-r} \leq \eps / 3$}.
	Selecting~\mbox{$k:= \min\left\{\bigl\lfloor
													12 \, \sqrt{d \, r} / \eps
												\bigr\rfloor,\,d\right\}$}
	then guarantees \mbox{$4 \,\sqrt{d \, r} / (k+1) \leq \eps / 3$},
	except for the case~$k=d$ where we can even ignore the second term
	from the estimate~\eqref{eq:error3summands}.
	Finally, the third term \mbox{$4 \cdot(\# A) / n$} can be bounded from above
	by~\mbox{$\eps / 3$} if we put
	\begin{equation*}
		n \,:=\, \left\lceil \frac{12}{\eps}
						\, \exp\left(k
											\left( 1
												+ \log\frac{d}{k}
												+ (\log 2)\, r
											\right)
									\right)
				\right\rceil
			\,\leq\, \exp\left[C \, \frac{\sqrt{d}}{\eps}
											\left(1 + \log \frac{d}{\eps}\right)^{3/2}\right] \,,
	\end{equation*}
	with some suitable numerical constant~$C > 0$.
	By this choice we obtain the error bound~$\eps$ we aimed for.
	
	Note that if~$\eps$ is too small, we can only choose~\mbox{$k = d$}
	for the algorithm~$A_{n,k,r}$.
	In this case, for the approximation of~$f_r$,
	we would take $2^{rd}$~wavelet coefficients into account,
	$n$~would become much bigger
	in order to achieve the accuracy we aim for.
	Instead, one can approximate~$f$ directly
	via the deterministic algorithm~$A_m^d$ from \thmref{thm:MonAppOrderConv},
	which is based on~\mbox{$(m-1)^d$} function values on a regular grid.
	The worst case error is bounded by~\mbox{$e(A_m^d,F_{\mon}^d) \leq d/m$}.
	Taking \mbox{$m := 2^r$},
	this gives the same bound
	that we already have for the accuracy at which~$\sgn f_r$ approximates~$f$,
	see~\eqref{eq:|f-fr|}.
	So for small~$\eps$, which roughly means $\eps \preceq 1/\sqrt{d}$
	(modulo logarithmic terms),
	we take the deterministic upper bound
	\begin{equation*}
		n^{\deter}(\eps,F_{\mon}^d)
			\,\leq\, \exp\left(d \, \log \frac{d}{\eps}
									\right) \,,
	\end{equation*}
	compare \remref{rem:MonAppOrderConv}.
\end{proof}

\begin{remark}[Violation of monotonicity]
	For the algorithms we analysed,
	there is no feature which would guarantee
	that the output function~$g$ is a monotonously increasing function.
	In fact, the analysis of \lemref{lem:monSmallWavelet}
	does only require that the function is monotone in each variable,
	but it is not necessary to know whether it is monotonously increasing
	or decreasing.
	
	We may think about a scenario
	where all computed function values are~$1$,
	but accidentally they are computed
	in the lowermost subcube~$C_{r\ones,\zeros}$ of the domain~$[0,1]^d$
	at resolution~$r$,
	and then some function values of the reconstruction~$g$ are still negative
	and violate the assumption of monotonic growth.
	Namely,
	for the linear reconstruction~$h$,
	the value in the uppermost subcube~$C_{r\ones,(2^r-1)\ones}$
	at resolution~$r$ can be written as
	\begin{equation*}
				\sum_{\substack{\vecalpha \in \{0,1\}^d \\
														|\vecalpha|_0 \leq k}} (-1)^{|\vecalpha|_0}
			\,=\,
				\sum_{\ell = 0}^k \binom{d}{\ell} (-1)^{\ell} \,.
	\end{equation*}
	If~$k \leq d/2$ is uneven, this value is negative.
	Meanwhile, $h$ is positive in~$C_{r\ones,\zeros}$,
	hence the monotonicity is violated, $g = \sgn h \notin F_{\mon}^d$.
\end{remark}

\begin{remark}[Implementation of the non-linear method $\bar{A}_{n,k,r}$]
	\label{rem:monoMCUBphicost}
	The algorithm~$\bar{A}_{n,k,r}$ as defined in~\eqref{eq:AlgF}
	appears rather abstract with the integral within the definition.
	There is an explicit way of representing the algorithm, though.
	Let~$\phi_{n,k,r}^{\omega}$ denote the mapping
	which returns the output~$h = \phi_{n,k,r}^{\omega}(\vecy)$
	for the linear algorithm~$A_{n,k,r}$ from~\eqref{eq:linAlg}
	when given the information
	$N^{\omega}(f) = \vecy = (y_1,\ldots,y_n) = (f(\vecX_1),\ldots,f(\vecX_n))$.
	For the reconstruction mapping~$\bar{\phi}_{n,k,r}$
	used in~$\bar{A}_{n,k,r}$ one may proceed as follows:
	\begin{itemize}
		\item Rearrange the information~\mbox{$(\vecX_1,y_1),\ldots,(\vecX_n,y_n)$}
			such that~\mbox{$y_1 \leq y_2 \leq \ldots \leq y_n$}.
		\item Define~\mbox{$y_0 := -1$} and~\mbox{$y_{n+1} := +1$},
			and use the representation
			\begin{equation*}
				\bar{\phi}_{n,k,r}^{\omega}(\vecy)
					\,=\, \frac{1}{2}
							\sum_{i=0}^{n}
								(y_{i+1}-y_i)
									\, \sgn \underbrace{
															\phi_{n,k,r}^{\omega}(\overbrace{-1,\ldots,-1
																											}^{\text{$i$ times}},
																						\overbrace{1,\ldots,1
																											}^{\text{$(n-i)$ times}})
														}_{=: g_i} \,.
			\end{equation*}
	\end{itemize}
	Implementing the nonlinear algorithm~$\bar{A}_{n,k,r}$
	is more difficult than for the linear algorithm,
	the cost for processing the collected information may exceed~$n$.
	There are different models of computation,
	see for example the book on IBC of Traub et al.~\cite[p.~30]{TWW88},
	or Novak and Wo\'zniakowski~\cite[Sec~4.1.2]{NW08}.
	Heinrich and Milla~\cite[Sec~6.2]{HeM11} point out
	that for problems with functions as output,
	the interesting question is not always about a complete picture of the
	output~\mbox{$\phi(\vecy)$},
	but about effective computation of approximate
	function values~\mbox{$[\phi(\vecy)](\vecx)$} on demand.
	It makes sense to distinguish
	between pre-processing operations and operations on demand.
	
	In our situation, pre-processing is concerned with
	rearranging the information,
	for which the expected computational cost is of order~$\mathcal{O}(n \log n)$.
	
	The main difficulty when asked to compute a function value~$g(\vecx)$
	on demand
	is to compute~$g_i(\vecx)$ for~$i=1,\ldots,n$.
	Once we know~$g_0(\vecx)$, it will be easier to
	compute~$g_1(\vecx),g_2(\vecx),\ldots$ in consecutive order
	because only few wavelet coefficients are affected when switching
	from $y_i = -1$ to $y_i = +1$.
	Namely, by linearity of~$\phi_{n,k,r}^{\omega}$ we have
	\begin{equation*}
		g_i := g_{i-1} - 2 \, \phi_{n,k,r}^{\omega}(\vece_i) \,,
	\end{equation*}
	where~$\vece_i = (\delta_{ij})_{j=1}^n$ is the $i$-th unit vector in~$\R^n$.
	Going back to the details of one-dimensional Haar wavelets~$\psi_{\alpha_j}$,
	observe that
	\begin{align*}
		[\phi_{n,k,r}^{\omega}(\vece_i)](\vecx)
			&\,=\, \frac{1}{n} \sum_{\substack{\vecalpha \in \N_0^d\\
												|\vecalpha|_0 \leq k\\
												\veclambda < r}}
					\psi_{\vecalpha}(\vecX_i) \, \psi_{\vecalpha}(\vecx)
			\,=\, \frac{1}{n} \sum_{\substack{\vecalpha \in \N_0^d\\
												|\vecalpha|_0 \leq k\\
												\veclambda < r}}
					\prod_{j=1}^d \psi_{\alpha_j}(\vecX_i(j)) \, \psi_{\alpha_j}(x_j) \\
			&\,=\, \frac{1}{n} \sum_{\substack{\vecbeta \in \{0,1\}^d\\
												|\beta|_1 \leq k}}
					\vecZ^{\vecbeta} \,,
	\end{align*}
	where~\mbox{$Z_j := \sum_{\alpha_j=1}^{2^r-1}
												\psi_{\alpha_j}(\vecX_i(j))
													\, \psi_{\alpha_j}(x_j)$}
	and~$\vecX_i(j)$ denotes the $j$-th entry of \mbox{$\vecX_i \in [0,1]^d$}.
	It is readily checked that
	\begin{equation*}
		Z_j = \begin{cases}
						2^r - 1
							\quad&\text{if $\lfloor 2^r \, \vecX_i(j) \rfloor
																= \lfloor 2^r \, x_j \rfloor$,}\\
						-1
							\quad&\text{else,}
					\end{cases}
	\end{equation*}
	so a comparison of the first~$r$ digits of the binary
	representation of~$\vecX_i(j)$ and $x_j$
	is actually enough for determining~$Z_j$.
	In the end,
	we only need the number~$b$ of coordinates~\mbox{$j \in \{1,\ldots,d\}$}
	where \mbox{$\lfloor 2^r \, \vecX_i(j) \rfloor
									= \lfloor 2^r \, x_j \rfloor$},
	and obtain
	\begin{equation*}
		n \, [\phi_{n,k,r}^{\omega}(\vece_i)](\vecx)
			\,=\, \sum_{\ell = 0}^{b \wedge k} \binom{b}{\ell} \, (2^r - 1)^{\ell} \,
					\sum_{m = 0}^{(d-b)\wedge(k-\ell)} \binom{d-b}{m} \, (-1)^m
			\;=:\; \chi(b) \in \Z\,.
	\end{equation*}
	These values~\mbox{$\chi(b)$}
	are needed for \mbox{$b \in \{0,\ldots,d\}$}.
	Since they only depend on parameters of the algorithm,
	they can be prepared before the algorithm is applied to an instance
	and do not count for the cost of processing the data.
	From these values one can also compute~$g_0(\vecx)$.
	If we do not want to store the values
	$[\phi_{n,k,r}^{\omega}(\vece_i)](\vecx)$ for \mbox{$i=1,\ldots,n$},
	we will need to compute them twice -- once in order to evaluate~$g_0(\vecx)$,
	once for calculating the difference between
	$g_i(\vecx)$ and $g_{i-1}(\vecx)$.
	Or we only compute them once but need storage for~$n$ numbers.
	In any case, the number of binary and fixed point operations
	needed for computing the output
	$g(\vecx) = [\bar{\phi}_{n,k,r}^{\omega}](\vecx)$ on demand
	is of order~$\mathcal{O}(drn)$.
\end{remark}

\section{Intractability of randomized approximation}
\label{sec:LBs}

\subsection{The result -- Monte Carlo lower bound}
\label{sec:LBs-Result}

As we have seen in \secref{sec:UBs},
for the $L_1$-approximation of monotone functions
the curse of dimensionality does not hold anymore in the randomized setting.
Within this section, however, we show a lower bound which implies that
the problem is still \emph{not weakly tractable} in the randomized setting,
we thus call it \emph{intractable}.

For the proof we switch to an average case setting for Boolean functions,
an idea that has already been used by
Blum, Burch, and Langford~\cite[Sec~4]{BBL98}.
From their result one can already extract
that for any fixed~\mbox{$\eps \in (0,1)$}
the Monte Carlo complexity for the approximation
of monotone functions depends at least exponentially on~$\sqrt{d}$.
The focus of Blum et al.\ was to show
that if we admit the information budget~$n$ to grow only polynomially in~$d$,
the achievable error will approach the initial error
at a rate of almost~\mbox{$1/\sqrt{d}$}.
In contrast, the aim of this paper is to obtain lower complexity bounds
for a range of error thresholds~$\eps$
which is moderately approaching zero as~$d$ is growing.
This enables us to prove \emph{intractability}.
The different focus leads to the necessity
of different tools within the corresponding lower bound proof,
see \cite[Remark~4.8]{Ku17} for a detailed discussion.

\begin{theorem} \label{thm:LB}
	Consider the randomized approximation of monotone functions.
	There exist constants~\mbox{$\sigma_0,\nu,\eps_0 > 0$}
	and \mbox{$d_0 \in \N$}
	such that for~\mbox{$d \geq d_0$} we have
	\begin{equation*}
		n^{\ran}(\eps_0,F_{\mon}^d)
			\,>\, \nu \, \exp(\sigma_0 \sqrt{d}) \,,
	\end{equation*}
	and moreover, for \mbox{$\eps_0 \, \sqrt{d_0/d} \leq \eps \leq \eps_0$}
	we have
	\begin{equation*}
		n^{\ran}(\eps,F_{\mon}^d)
			\,>\, \nu \, \exp\left(c \, \frac{\sqrt{d}}{\eps}\right) \,,
	\end{equation*}
	with $c = \sigma_0 \, \eps_0$ \,.
	
	Specifically, for~\mbox{$d \geq d_0 = 100$}
	and \mbox{$\eps_0 = \frac{1}{15}$} we have
	\begin{equation*}
		n^{\ran}({\textstyle \frac{1}{15}},F_{\mon}^d)
			\,>\, 108 \cdot \exp(\sqrt{d}-\sqrt{100}) \,,
	\end{equation*}
	for~\mbox{$d=100$} this means \mbox{$n^{\ran} > 108$}.
	For~\mbox{$\frac{1}{15} \sqrt{100/d} \leq \eps \leq \frac{1}{15}$}
	we have
	\begin{equation*}
		n^{\ran}(\eps,F_{\mon}^d)
			\,>\, 108 \cdot \exp\left(\frac{\sqrt{d}}{15 \, \eps} -\sqrt{100}\right)
			\,.
	\end{equation*}
\end{theorem}
All these lower bounds hold for varying cardinality as well,
see~\cite[Section~4.3]{Ku17}.
Before we give the proof in \secref{sec:monoMCLBs-Proof},
we discuss a theoretical consequence of the theorem.

\begin{remark}[Intractability] \label{rem:monMCLBintractable}
	The above theorem shows that the approximation of monotone functions
	is \emph{not weakly tractable}.
	Indeed,	consider the sequence~\mbox{$(\eps_d)_{d=d_0}^{\infty}$}
	of error thresholds~\mbox{$\eps_d := \eps_0 \, \sqrt{d_0 / d}$}.
	Then, regarding~$n^{\ran}(\eps,F_{\mon}^d)$ as a function~$n(\eps,d)$,
	we observe
	\begin{equation*}
		\limsup_{\eps^{-1}+d \to \infty} \frac{\log n(\eps,d)}{\eps^{-1}+d}
			\,\geq\, \lim_{d \rightarrow \infty}
							\frac{\log n(\eps_d,d)}{\eps_d^{-1} + d}
			\,\geq\, \lim_{d \rightarrow \infty}
					\frac{\sigma_0 \, d /\sqrt{d_0} + \log \nu
							}{\eps_0^{-1} \sqrt{d/d_0} + d}
			\,=\, \frac{\sigma_0}{\sqrt{d_0}} > 0 \,.
	\end{equation*}
	This contradicts the definition of weak tractability,
	as defined in the book of Novak and Wo\'zniakowski~\cite{NW08}.
	Namely, the problem would be called \emph{weakly tractable}
	if the limit superior was zero.
	We can also put it like this:
	in our situation $n(\eps_d,d)$ grows exponentially in~$d$
	despite the fact that~$\eps_d$ is only moderately decreasing.
	
	Actually, this behaviour has already been known since the paper
	of Bshouty and Tamon 1996~\cite[Thm~5.3.1]{BT96}
	on Boolean monotone functions,
	however, research on weak tractability has not yet been started
	at that time.
	Their lower bound can be summarized as follows:
	For moderately decaying error
	thresholds~\mbox{$\eps_d \preceq (\sqrt{d} \, (1+ \log d))^{-1}$}
	and sufficiently large~$d$, we have
	\begin{equation*}
		n^{\ran}(\eps_d,F_{\mon}^d)
			\,\geq\, c \, 2^d / \sqrt{d} \,,
	\end{equation*}
	with some numerical constant~$c > 0$.
	Interestingly, the proof is based on purely combinatorial arguments,
	without applying minimax arguments.
	From their approach, however, we can only derive a statement
	for smaller and smaller~$\eps$ as~\mbox{$d \rightarrow \infty$}.
	So the new lower bounds indeed give a more complete picture
	on the joint $(\eps,d)$-dependency of the complexity.
	Since the proof of our theorem is based on Boolean functions,
	actually we have lower bounds for the easier problem
	of approximating Boolean monotone functions.
\end{remark}

\subsection{The proof of the Monte Carlo lower bound}
\label{sec:monoMCLBs-Proof}

This section contains the proof of \thmref{thm:LB}.
Key ideas have already been used by Blum et al.~\cite{BBL98},
albeit only in the context of Boolean functions.
Some modifications within the present proof are mere simplifications
with the side effect of improved constants,
but several changes are substantial and marked as such.

We consider the subclass~$F_2^d \subset F_{\mon\pm}^d$
of sign-valued monotone functions
which are constant on each of the~$2^d$ subcubes~$C_{\veci}$,
$\veci \in \{0,1\}^d$,
if we split the domain~$[0,1]^d$
just as in the proof of \thmref{thm:MonAppOrderConv} with~$m=2$.
Any such function~$f \in F_2^d$ is uniquely determined
by its function values~$f(\vecx)$ in the corners~$\vecx \in \{0,1\}^d$,
we have~$f|_{C_{\vecx}} = f(\vecx)$,
so effectively we deal with Boolean functions.
The lower bound proof for general Monte Carlo methods
relies on \emph{Bakhvalov's trick}~\cite{Bakh59},
compare the lower bound part within the proof of~\thmref{thm:MonAppOrderConv}
for a more basic example of this proof technique.
Here now, we construct a probability measure~$\mu$ on~$F_2^d$
and use that for any Monte Carlo algorithm~\mbox{$(A_n^{\omega})$}
we have
\begin{align}
	e((A_n^{\omega}),F_{\mon}^d)
		&\,=\, \sup_{f \in F_{\mon}^d} \expect \|A_n^{\omega} - f\|_{L_1}
		\,\geq\, \int \expect \|A_n^{\omega}(f) - f\|_{L_1} \dint\mu(f)
		\nonumber\\
	\text{[Fubini]}\qquad
		&\,=\, \expect \int \|A_n^{\omega}(f) - f\|_{L_1} \dint\mu(f)
		\,\geq\, \inf_{A_n} \underbrace{\int \|A_n(f) - f\|_{L_1} \dint\mu(f)
																	}_{=: e(A_n,\mu)}
		\label{eq:Bakh}\,,
\end{align}
where the infimum runs over all deterministic algorithms
$A_n = \phi \circ N$ that use at most~$n$ function values.
In order to construct optimal algorithms~$A_n$
with regard to minimizing the so-called
\emph{$\mu$-average error}~$e(A_n,\mu)$,
one will need to optimize	the output function~$g=\phi(\vecy)$
with respect to the conditional measure~$\mu_{\vecy}$
after knowing information~$\vecy := N(f)$.
In our specific situation, which is the \mbox{$L_1$-approximation}
of sign-valued functions, the optimal output is sign-valued as well,
taking the pointwise conditional median.
The conditional error for this optimal output
is given by
\begin{multline}
	\inf_{g \in L_1} \int \|f-g\|_{L_1} \dint\mu_{\vecy}(f)
		\,=\, 2 \int_{[0,1]^d} \min\left\{\mu_{\vecy}\{f(\vecx)=-1\},
																				\mu_{\vecy}\{f(\vecx)=+1\}
																\right\}
							\dint\vecx
			\\
		\,=\, 2^{1-d} \sum_{\vecx \in \{0,1\}^d} 
											\min\left\{\mu_{\vecy}\{f(\vecx)=-1\},
																	\mu_{\vecy}\{f(\vecx)=+1\}
													\right\}
			\label{eq:muyoptg} \,.
\end{multline}
We will further use the concept of \emph{augmented information}
$\tilde{y} = \widetilde{N}(f)$ which contains additional knowledge
on the input compared to the original information~$\vecy = N(f)$.
This will lead to more powerful algorithms with smaller errors,
but it is done for the sake of an easier description
of the corresponding conditional measure~$\mu_{\tilde{y}}$.
The lower bounds we obtain for algorithms with the augmented oracle,
a fortiori, are lower bounds for algorithms with the standard oracle.

The proof is organized in seven steps.
\begin{description}
	\item[{\proofstepref{proof:monoLB1}}:]
		The general structure of the measure~$\mu$ on~$F_2^d$.
	\item[{\proofstepref{proof:monoLB2}}:]
		Introduce the augmented information.
	\item[{\proofstepref{proof:monoLB3}}:]
		Estimate the number of points~$\vecx \in \{0,1\}^d$
		for which $f(\vecx)$ is still -- to some extend -- undetermined,
		even after knowing the augmented information.
	\item[{\proofstepref{proof:monoLB4}}:]
		Further specify the measure~$\mu$,
		and give estimates on the conditional probability
		for the event~\mbox{$f(\vecx) = -1$} for the set of still fairly
		uncertain~$\vecx$ from the step before.
	\item[{\proofstepref{proof:monoLB5}}:]
		A general formula for the lower bound.
	\item[{\proofstepref{proof:monoLB6}}:]
		Relate estimates for~$\eps_0$ and~$d_0$ to estimates
		for smaller~$\eps$ and larger~$d$.
	\item[{\proofstepref{proof:monoLB7}}:]
		Explicit numerical values.
\end{description}
	
	\proofstep{proof:monoLB1}{%
		General structure of the measure~$\mu$.}
	We define a measure $\mu$ on~$F_2^d$
	that can be represented by a randomly drawn
	set~\mbox{$U \subseteq W := \{\vecx \in \{0,1\}^d \mid |\vecx|_1 = t\}$},
	with \mbox{$t \in \N$} being a suitable parameter, and a boundary
	value~\mbox{$b \in \N$}, \mbox{$t \leq b \leq d$}, namely
	\begin{equation} \label{eq:f_U}
		f_{U}(\vecx)
			:=\begin{cases}
					-1, \qquad&\text{if \,$|\vecx|_1 \leq b$\,
													and $\nexists\, \vecu \in U$
														with \,$\vecu \leq \vecx$\,,} \\
					+1, \qquad&\text{if \,$|\vecx|_1 > b$\, or
														$\exists\, \vecu \in U$
														with \,$\vecu \leq \vecx$\,.}
				\end{cases}
	\end{equation}
	The boundary value~\mbox{$b \in \N$}
	will facilitate considerations in connection with the augmented
	information in \proofstepref{proof:monoLB2}.
	We draw~$U$ such that the~\mbox{$f(\vecw)$} with~$\vecw \in W$ are independent
	Bernoulli random variables
	with~\mbox{$p = \mu\{f(\vecw) = +1\} = 1 - \mu\{f(\vecw) = -1\}$}.
	The parameter~\mbox{$p \in (0,1)$}
	will be specified in \proofstepref{proof:monoLB4}.
	
	\proofstep{proof:monoLB2}{%
		Augmented information.}
	Now, for any (possibly adaptively obtained)
	info~\mbox{$\vecy = N(f) = (f(\vecx_1),\ldots,f(\vecx_n))$}
	with~\mbox{$\vecx_i \in \{0,1\}^d$}, we define the augmented information
	\begin{equation*} 
		\tilde{y} := (V_{\ominus},V_{\oplus}),
	\end{equation*}
	where~\mbox{$V_{\ominus} \subseteq W \setminus U$} and~\mbox{$V_{\oplus} \subseteq U$}
	represent knowledge about the instance~$f$ that $\mu$-almost surely
	implies the information~$\vecy$.
	We know \mbox{$f(\vecu) = -1$ for $\vecu \in V_{\ominus}$},
	and \mbox{$f(\vecu) = +1$ for $\vecu \in V_{\oplus}$}.
	In detail,
	let $\leq_{\text{L}}$ be the lexicographic order
	of the elements of~$W$,
	then \mbox{$\min_{\text{L}} V$} denotes the first element
	of a set~$V \subseteq W$ with respect to this order.
	For a single function evaluation~$f(\vecx)$
	the augmented oracle reveals the sets
	\begin{align*}
			V_\ominus^{\vecx}
				&:= \begin{cases}
							\emptyset
									&\quad\text{if \,$|\vecx|_1 > b$\,,}\\
							\{\vecv \in W \mid \vecv \leq \vecx\}
									&\quad\text{if \,$f(\vecx) = -1$\,,}\\
							\{\vecv \in W
									\mid
									\vecv \leq \vecx
										\text{ and }
											\vecv <_{\text{L}}
													{\textstyle \min_{\text{L}}}
														\{\vecu \in U \mid
															\vecu \leq \vecx\}
							\}
									&\quad\text{if \,$f(\vecx) = +1$\,} \\
									&\quad\text{and \,$|\vecx|_1 \leq b$\,,}\\
						\end{cases}\\
			V_{\oplus}^{\vecx}
				&:= \begin{cases}
							\emptyset
									&\quad\text{if \,$|\vecx|_1 > b$\,
															or \,$f(\vecx) = -1$\,,}\\
							\{{\textstyle\min_{\text{L}}}
									\{\vecu \in U \mid
										\vecu \leq \vecx\}
							\}
									&\quad\text{if \,$f(\vecx) = +1$\,
															and \,$|\vecx|_1 \leq b$\,,}
						\end{cases}
	\end{align*}
	and altogether the augmented information is
	\begin{equation*}
		\tilde{y} = (V_{\ominus},V_{\oplus})
			:=\left(\bigcup_{i=1}^n V_{\ominus}^{\vecx_i}
							\, , \,
							\bigcup_{i=1}^n V_{\oplus}^{\vecx_i}
				\right) \,.
	\end{equation*}
	Note that computing~$f(\vecx)$ for~\mbox{$|\vecx|_1 > b$}
	is a waste of information,
	so no algorithm designer would decide to compute such samples.
	Since~\mbox{$\# V_{\ominus}^{\vecx} \leq \binom{|\vecx|_1}{t}
																	 \leq \binom{b}{t}$}
	for~\mbox{$|\vecx|_1 \leq b$},
	and~\mbox{$\# V_{\oplus}^{\vecx} \leq 1$},
	we have the estimates
	\begin{equation}\label{eq:|V01|bound}
		\# V_{\ominus} \leq n \, \binom{b}{t} \,,
		\quad\text{and}\quad
		\# V_{\oplus} \leq n \,.
	\end{equation}
	(Blum et al.~\cite{BBL98} did not have a boundary value~$b$
	but used a Chernoff bound for giving a probabilistic bound on~$\# V_{\ominus}$
	instead.)
	
	
	\proofstep{proof:monoLB3}{%
		Number of points~\mbox{$\vecx \in \{0,1\}^d$}
		where
		\mbox{$f(\vecx)$} is still fairly uncertain.}
	For any point~\mbox{$\vecx \in \{0,1\}^d$} we define the set
	\begin{equation*}
		W_{\vecx} := \{\vecw \in W \mid
												\vecw \leq \vecx\} \,
	\end{equation*}
	of points that are ``relevant'' to~\mbox{$f(\vecx)$}.
	Given the augmented information
	\mbox{$\tilde{y} = (V_{\ominus},V_{\oplus})$},
	we are interested in points where it is not yet clear
	whether~\mbox{$f(\vecx) = +1$} or~\mbox{$f(\vecx) = -1$}.
	In detail, these are points~$\vecx$ where~\mbox{$W_{\vecx} \cap V_{\oplus} = \emptyset$},
	for that \mbox{$f(\vecx) = -1$} be still possible.
	Furthermore, \mbox{$W_{\vecx} \setminus V_{\ominus}$} shall be big enough,
	say \mbox{$\#(W_{\vecx} \setminus V_{\ominus}) \geq M$} with~\mbox{$M \in \N$},
	so that the conditional probability~%
	\mbox{$p_{\vecx} := \mu_{\tilde{y}}\{f(\vecx) = +1\}$}
	is not too small.
	For our estimates in \eqref{eq:Markov W_x-V_->M}
	it will be necessary to restrict
	to points~\mbox{$|\vecx|_1 \geq a \in \N$},
	we suppose~\mbox{$t \leq a \leq b$}.
	The set of all these points shall be denoted by
	\begin{align*}
			B &:= \{\vecx \in D_{ab} \mid
							W_{\vecx} \cap V_{\oplus} = \emptyset ,\,
							\#(W_{\vecx} \setminus V_{\ominus}) \geq M
						\} \,,\\
			\text{where} \quad
			D_{ab} &:= \{\vecx \in \{0,1\}^d \mid
										a \leq |\vecx|_1 \leq b\}\,.
	\end{align*}
	We aim to find a lower bound for the cardinality of $B$.
	Within the proof of Blum et al.~\cite{BBL98} Hoeffding bounds were used.
	We will employ the Berry-Esseen inequality
	on the speed of convergence of the Central Limit Theorem, instead,
	and it is only with Berry-Esseen
	that we can draw conclusions for small~$\eps$,
	as it is done in \proofstepref{proof:monoLB6}.

\begin{proposition}[Berry-Esseen inequality] \label{prop:BerryEsseen}
	Let \mbox{$Z_1,Z_2,\ldots$} be i.i.d. random variables with
	zero mean, unit variance and finite third absolute moment~$\beta_3$.
	Then there exists a universal constant~$C_0$
	such that
	\begin{equation*}
		\left|\P\biggl\{\frac{1}{\sqrt{d}} \sum_{j=1}^d Z_j \leq x \biggr\}
					- \Phi(x)\right|
			\,\leq\, \frac{C_0 \, \beta_3}{\sqrt{d}} \,,
	\end{equation*}
	where $\Phi(\cdot)$ is the cumulative distribution function of the
	univariate standard normal distribution.\\
	The best known estimates on~$C_0$ are
	\begin{equation*}
		C_E := \frac{\sqrt{10}+3}{6\sqrt{2\pi}} = 0.409732\ldots
			\leq C_0
			< 0.4748
	\end{equation*}
	see Shevtsova~\cite{Shev11}.
\end{proposition}


\proofsubstep{proof:monoLB3.1}{%
		Bounding~$\# D_{a,b}$.}
	Let~\mbox{$a := \lceil \frac{d}{2} + \alpha \frac{\sqrt{d}}{2} \rceil$}
	and~\mbox{$b := \lfloor \frac{d}{2} + \beta \frac{\sqrt{d}}{2} \rfloor$}
	with~\mbox{$\beta - \alpha \geq 2/\sqrt{d} $}, hence~$a \leq b$.
	Consider Rademacher random variables
	\mbox{$Z_1,\ldots,Z_d \stackrel{\text{iid}}{\sim} \Uniform \{-1,+1\}$}.
	Note that the~$Z_j$ have zero mean, unit variance,
	and third absolute moment~\mbox{$\beta_3 = 1$}.
	Applying \propref{prop:BerryEsseen} twice to the~$Z_j$, we obtain
	\begin{align} 
		\frac{\# D_{ab}}{\# \{0,1\}^d}
			&\,=\, \frac{1}{2^d} \sum_{k=a}^{b} \binom{d}{k}
			\,=\, \P\biggl\{\alpha \leq \frac{1}{\sqrt{d}}
							\, \sum_{j=1}^d Z_j \leq \beta \biggr\}
			\nonumber\\
			&\,\geq\, \underbrace{\Phi(\beta) - \Phi(\alpha)
													}_{=: C_{\alpha\beta}}
							 - \frac{2 \, C_0}{\sqrt{d}}
			\,=:\, r_0(\alpha,\beta,d) \,.\label{eq:card(D_ab)}
	\end{align}
	
	\proofsubstep{proof:monoLB3.2}{%
		The influence of $\vecw \in W$ (in particular~$\vecw \in V_{\oplus}$).}
	(This step becomes essential
	for small~$\eps$ in \proofstepref{proof:monoLB6}.
	For the focus of Blum et al.~\cite{BBL98}
	with $\eps$ close to the initial error,
	the trivial estimate
	\mbox{$\# (Q_{\vecw} \cap D_{ab}) / \# Q_{\vecw} \leq 1$}
	suffices.)
	Now, let~\mbox{$t := \lceil \tau \sqrt{d} \rceil$}
	with~\mbox{$\tau > 0$}, and for~\mbox{$\vecw \in W$} define
	\begin{equation*}
		Q_{\vecw} := \{\vecx \in \{0,1\}^d \mid
												\vecw \leq \vecx\} \,,
	\end{equation*}
	this is the set of all points inside the area of influence of~$\vecw$.
	Similarly to \proofstepref{proof:monoLB3.1}, we obtain
	\begin{align}
			\frac{\# (Q_{\vecw} \cap D_{ab})}{\# Q_{\vecw}}
				&\,= \, \frac{\# \{\vecx \in \{0,1\}^{d-t} \mid
														a-t \leq |\vecx|_1 \leq b-t\}
											}{2^{d-t}}
					\nonumber\\
				&\,= \, \frac{1}{2^{d-t}} \sum_{k=a-t}^{b-t} \binom{d-t}{k}
					\nonumber\\
			[\text{\propref{prop:BerryEsseen}}]
			\quad &\,\leq \,
							\Phi\left(\frac{2b - t}{\sqrt{d-t}}\right)
										- \Phi\left(\frac{2a - t}{\sqrt{d-t}}\right)
										+\frac{2 \, C_0}{\sqrt{d-t}}
					\nonumber\\
			[\text{\eqref{eq:Gauss-stretch}, \eqref{eq:Gauss-shift}}]
			\quad&\,\leq \,
							\biggl[\underbrace{\Phi\left(\beta-\tau\right)
																- \Phi\left(\alpha - \tau\right)
																}_{=: C_{\alpha\beta\tau}}
											+\underbrace{\left(\frac{1}{\sqrt{2 \pi}} + 2 \, C_0\right)
																	}_{=: C_1}
													\frac{1}{\sqrt{d}}
								\biggr]
								\, \frac{1}{\sqrt{1 - t/d}}\,,
			\label{eq:Q_w-influenced}
	\end{align}
	where for $\tau < \sqrt{d} - 1/\sqrt{d}$ we have
	\begin{equation} \label{eq:k_t(d)}
		1 \,\leq\, \frac{1}{\sqrt{1-t/d}}
			\,\leq\, \frac{1}{\sqrt{1 - \tau / \sqrt{d} - 1/d}}
			\,=:\, \kappa_{\tau}(d)
			\,\xrightarrow[d\rightarrow\infty]{}\, 1 \,.
	\end{equation}
	Within the above calculation~\eqref{eq:Q_w-influenced},
	we exploited that the density of the Gaussian distribution
	is decreasing with growing distance to~$0$,
	in detail, for $t_0<t_1$ and $\kappa \geq 1$ we have
	\begin{align}
		\Phi(\kappa \, t_1) - \Phi(\kappa \, t_0)
			&\,=\, \frac{1}{\sqrt{2\pi}} \int_{\kappa \, t_0}^{\kappa \, t_1}
				\exp\left(-\frac{t^2}{2}\right) \dint t
			\,=\, \frac{\kappa}{\sqrt{2\pi}} \int_{t_0}^{t_1}
				\exp\left(-\frac{\kappa^2 \, s^2}{2}\right) \dint s
			\nonumber\\
			&\,\leq\, \frac{\kappa}{\sqrt{2\pi}} \int_{t_0}^{t_1}
				\exp\left(-\frac{s^2}{2}\right) \dint s
			\,=\, \kappa \left[\Phi(t_1) - \Phi(t_0)\right] \,.
			\label{eq:Gauss-stretch}
	\end{align}
	Namely, we took~\mbox{$\kappa = 1/\sqrt{1 - t/d}$}
	which comes from replacing~\mbox{$1/\sqrt{d-t}$} by~\mbox{$1/\sqrt{d}$}.
	Furthermore, we shifted the \mbox{$\Phi$-function},
	knowing that its derivative takes values
	between~$0$ and~\mbox{$1/\sqrt{2\pi}$},
	so for~\mbox{$t_0<t_1$} and \mbox{$\delta \in \R$} we have
	\begin{equation} \label{eq:Gauss-shift}
		\Bigl|\bigl[\Phi(t_1 + \delta) - \Phi(t_0 + \delta)\bigr]
								-\bigl[\Phi(t_1) - \Phi(t_0)\bigr]
					\Bigr|
			\,\leq\, \frac{|\delta|}{\sqrt{2\pi}} \,,
	\end{equation}
	in our case \mbox{$\delta = t / \sqrt{d} - \tau
												\leq 1 / \sqrt{d}$}.
	
	\proofsubstep{proof:monoLB3.3}{%
		The influence of~$V_{\ominus}$.}
	Markov's inequality gives
	\begin{equation} \label{eq:V_-Markov}
		\sum_{\vecw \in V_{\ominus}} \# (Q_{\vecw} \cap D_{ab})
			\,=\, \sum_{\vecx \in D_{ab}} \# (W_{\vecx} \cap V_{\ominus})
			\,\geq\, N \, \#\{\vecx \in D_{ab} \mid
										\#(W_{\vecx} \cap V_{\ominus}) \geq N\} \,,
	\end{equation}
	with~\mbox{$N \in \N$}.
	Using this, we can carry out the estimate
	\begin{equation} \label{eq:Markov W_x-V_->M}
		\begin{split}
			\# \{\vecx \in D_{ab} \mid
						\#(W_{\vecx} \setminus V_{\ominus}) \geq M \}
				&\,=\, \# \{\vecx \in D_{ab} \mid
										\#(W_{\vecx} \cap V_{\ominus}) \leq \# W_{\vecx} - M \} \\
				&\,\geq\,
					\# \{\vecx \in D_{ab} \mid
							\#(W_{\vecx} \cap V_{\ominus}) \leq {\textstyle\binom{a}{t}} - M \} \\
				&\,=\,
					\# D_{ab}
					- \# \{\vecx \in D_{ab} \mid
								\#(W_{\vecx} \cap V_{\ominus}) > {\textstyle\binom{a}{t}} - M \} \\
			\text{[\eqref{eq:V_-Markov}]}\quad
				&\,\geq\,
					\# D_{ab}
					- \frac{1}{{\textstyle \binom{a}{t}} - M + 1}
							\, \sum_{\vecw \in V_{\ominus}} \# (Q_{\vecw} \cap D_{ab}) \,.
		\end{split}
	\end{equation}
	
	\proofsubstep{proof:monoLB3.4}{%
		Final estimates on~$\# B$.}
	Putting all this together, we estimate the cardinality of~$B$:
	\begin{align*}		
			\frac{\# B}{\#\{0,1\}^d}
				&\,=\, \frac{\# \left(\{\vecx \in D_{ab} \mid
														\#(W_{\vecx} \setminus V_{\ominus}) \geq M \}
													\setminus \bigcup_{\vecw \in V_{\oplus}} Q_{\vecw}
										\right)
								}{\#\{0,1\}^d} \\
				\text{[\eqref{eq:Markov W_x-V_->M}, any $\vecw \in W$]}\quad
				&\,\geq\,
					\frac{\# D_{ab}}{\#\{0,1\}^d} \\
				&\qquad - \frac{\# Q_{\vecw}}{\# \{0,1\}^d}
								\, \left(\frac{\# V_{\ominus}}{\binom{a}{t} - M + 1} + \#V_{\oplus}\right)
								\, \frac{\# (Q_{\vecw} \cap D_{ab})}{\# Q_{\vecw}}\\
			\text{[\eqref{eq:|V01|bound},
						 \eqref{eq:card(D_ab)},
						 \eqref{eq:Q_w-influenced}+\eqref{eq:k_t(d)}]}\quad
			&\,\geq\,
				C_{\alpha\beta} - \frac{2 \, C_0}{\sqrt{d}}\\
				&\qquad
				- n \, 2^{-t} \, \left(\frac{\binom{b}{t}}{\binom{a}{t} - M + 1} + 1\right)
					\, \left[C_{\alpha\beta\tau}
										+\frac{C_1}{\sqrt{d}}
						\right]
						\, \kappa_{\tau}(d) \,.
	\end{align*}
	Assuming \mbox{$\alpha - 2\tau \geq - \sqrt{d} + 2/\sqrt{d}$}
	will guarantee \mbox{$t < a$}.
	We estimate the ratio
	\begin{align} 
			\binom{b}{t}\bigg/\binom{a}{t}
				&\,\leq\, \left(\frac{a+1}{a-t+1}\right)^{b-a}
				\,\leq\, \left(\frac{\frac{d}{2} + \alpha \frac{\sqrt{d}}{2} + 1
												}{\frac{d}{2} + (\alpha - 2\tau) \frac{\sqrt{d}}{2}}
							\right)^{(\beta-\alpha)\,\sqrt{d}/2}
					\nonumber\\
				&\,\leq\, \exp\Biggl((\beta - \alpha) \, \tau
												\underbrace{\left(1 + \frac{\alpha - 2 \tau}{\sqrt{d}}
																		\right)^{-1}
																	}_{=: \kappa_{\alpha\tau}(d)}
												+\underbrace{\frac{\beta-\alpha
																					}{\sqrt{d} + \alpha - 2\tau}
																		}_{=: K_{\alpha\beta\tau}(d)}
									\Biggr)
				\,=:\, \sigma_{\alpha\beta\tau}(d) \,,
				\label{eq:sigma_abt(d)}
	\end{align}
	where we have~\mbox{$1 \leq \kappa_{\alpha\tau}(d)
														\xrightarrow[d\rightarrow\infty]{} 1$}
	and \mbox{$0 \leq K_{\alpha\beta\tau}(d) \xrightarrow[d\rightarrow\infty]{} 0$}.
	(Note that the above estimate is asymptotically optimal,
	\mbox{$1 \leq \binom{b}{t}/\binom{a}{t}
										\xrightarrow[d\rightarrow\infty]{}
											\exp\left((\beta-\alpha)\,\tau\right)$}.)
	We finally choose
	the information cardinality~\mbox{$n = \lfloor\nu 2^t \rfloor$},
	and put~\mbox{$M := \lceil \lambda \binom{a}{t} \rceil$}
	with~\mbox{$0<\lambda<1$},
	so that we obtain the estimate
	\begin{align}
		\frac{\# B}{\#\{0,1\}^d}
			&\,\geq\, \underbrace{\left[C_{\alpha\beta} - \frac{2 \, C_0}{\sqrt{d}} 
												\right]
											}_{= r_0(\alpha,\beta,d)}
					- \nu \, \underbrace{\left(\frac{\sigma_{\alpha\beta\tau}(d)
																				}{1-\lambda} + 1
															\right)
																\, \left[C_{\alpha\beta\tau}
																					+\frac{C_1}{\sqrt{d}}
																	\right]
																\, \kappa_{\tau}(d)
															}_{=: r_1(\alpha,\beta,\tau,\lambda,d)}
				\nonumber\\
			&\;=:\; r_B(\alpha,\beta,\tau,\lambda,\nu,d) \,.
				\label{eq:|B|>=r0-nu*r1}
	\end{align}
	With all the other conditions on the parameters imposed before,
	for sufficiently large~$d$ we will have~\mbox{$r_0(\ldots) > 0$}.
	Furthermore, we always have~\mbox{$r_1(\ldots) > 0$},
	so choosing~\mbox{$0 < \nu < r_0(\ldots) / r_1(\ldots)$}
	will guarantee $r_B(\ldots)$ to be positive.
		
	\proofstep{proof:monoLB4}{%
		Specification of~$\mu$ and bounding conditional probabilities.}
	We specify the measure~$\mu$
	on the set of functions~\mbox{$\{f_U \mid U \subseteq W\} \subset F_2^d$}
	defined as in~\eqref{eq:f_U}.
	Recall that the~\mbox{$f(\vecw)$} (for~\mbox{$\vecw \in W$})
	shall be independent Bernoulli random variables with
	probability~\mbox{$p = \mu\{f(\vecw) = +1\}$}.
	Knowing the augmented information~%
	\mbox{$\tilde{y} = (V_{\ominus},V_{\oplus})$},
	the values~\mbox{$f(\vecw)$} are still independent random variables
	with conditional probabilities
	\begin{equation*}
		\mu_{\tilde{y}}\{f(\vecw) = +1\} =
			\begin{cases}
				0 &\text{if $\vecw \in V_{\ominus}$,} \\
				1 &\text{if $\vecw \in V_{\oplus}$,} \\
				p &\text{if $\vecw \in W \setminus (V_{\ominus} \cup V_{\oplus})$.}
			\end{cases}
	\end{equation*}
	Then for~\mbox{$\vecx \in B$} we have the estimate
	\begin{equation*}
		\mu_{\tilde{y}}\{f(\vecx) = -1\}
			\,\leq\, (1 - p)^M
			\,\leq\, \exp\left( - p \lambda \binom{a}{t}\right)
			\,=\, \exp( - \lambda \varrho) \,,
	\end{equation*}
	where we write~\mbox{$p := \varrho/\binom{a}{t}$}
	with~\mbox{$0 < \varrho < \binom{a}{t}$}.
	The other estimate is
	\begin{multline} \label{eq:q_0}
		\mu_{\tilde{y}}\{f(\vecx) = -1\}
			\,\geq\, (1 - p)^{\binom{b}{t}}
			\,=\, \exp\left(\log(1 - p) \, \binom{b}{t}\right) \\
			\,\geq\,
				\exp\Biggl(
							- \varrho \, \sigma_{\alpha\beta\tau}(d) \,
								\biggl(
									\underbrace{\frac{1}{2}
															+ \frac{1}{2\,(1- \varrho/\gamma_{\alpha\tau}(d))}
														}_{=: \kappa_{\varrho\gamma}(d)}
								\biggr)
						\Biggr)
			\,=:\, q_0(\alpha,\beta,\tau,\varrho,d) \\
			\,\xrightarrow[d \rightarrow \infty]{}\,
				\exp\Bigl(-\varrho \,
										\exp\left((\beta-\alpha)\, \tau\right)
						\Bigr)\,.
	\end{multline}
	Here we used that, for $0 \leq p < 1$,
	\begin{equation*}
		0 \,\geq\, \log(1-p)
			\,=\, - \left(p+\sum_{k=2}^{\infty} \frac{p^k}{k}\right)
			\,\geq\, - \left(p+\sum_{k=2}^{\infty} \frac{p^k}{2}\right)
			\,=\, -p\,\left(\frac{1}{2} + \frac{1}{2\,(1-p)}\right) \,,
	\end{equation*}
	together with the estimates
	\begin{equation*}
		p \, \binom{b}{t} \,\leq\, \varrho \, \sigma_{\alpha\beta\tau}(d) \,,
	\end{equation*}
	and
	\begin{equation} \label{eq:gamma}
		\binom{a}{t}
			\,\geq\, \left(\frac{a}{t}\right)^t
			\,\geq\,
				\left(\frac{\sqrt{d}+\alpha}{2\left(\tau + 1 /\sqrt{d}\right)}
				\right)^{\tau \sqrt{d}}
			\,=:\, \gamma_{\alpha\tau}(d) \,\geq\, 1 \,.
	\end{equation}
	The last estimate~\eqref{eq:gamma} relies on the constraint
	\mbox{$\alpha - 2\tau \geq - \sqrt{d} + 2/\sqrt{d}$}
	(and hence \mbox{$t < a$}).
	Note that \mbox{$\gamma_{\alpha\tau}(d)
								\xrightarrow[d \rightarrow \infty]{} \infty$}
	implies
	\mbox{$\kappa_{\varrho\gamma}(d)
								\xrightarrow[d \rightarrow \infty]{} 1$}.
	It follows that for~\mbox{$\vecx \in B$},
	\begin{multline} \label{eq:def_q}
		\min\Bigl\{\mu_{\tilde{y}}\{f(\vecx) = +1\},\,
							\mu_{\tilde{y}}\{f(\vecx) = -1\}
				\Bigr\}\\
			\,\geq\, \min\left\{1- \exp\left( - \varrho \, \lambda\right), \,
										q_0(\alpha,\beta,\tau,\varrho,d)
								\right\}
			\,=:\, q(\alpha,\beta,\tau,\lambda,\varrho,d) \,.
	\end{multline}
	
	\proofstep{proof:monoLB5}{%
		The final error bound.}
	By Bakhvalov's trick~\eqref{eq:Bakh}
	we obtain the final estimate
	for~\mbox{$n \leq \nu \, 2^{\tau\sqrt{d}} = \nu \, \exp(\sigma\sqrt{d})$},
	where~\mbox{$\sigma = \tau \, \log 2$},
	\begin{align}
		e^{\ran}(n,F_{\mon}^d)
			&\,\geq\, \inf_{A_n} e(A_n,\mu)
				\nonumber\\
		[\text{\eqref{eq:muyoptg} for any valid $\tilde{y}$}]\quad
			&\,\geq\, 2\, \frac{\# B}{\#\{0,1\}^d}
				\, \min\bigl\{\mu_{\tilde{y}}\{f(\vecx) = 0\},\,
											\mu_{\tilde{y}}\{f(\vecx) = 1\}
									\mid \vecx \in B
								\bigr\} 
				\nonumber\\
		[\text{\eqref{eq:|B|>=r0-nu*r1} and \eqref{eq:def_q}}]\quad
			&\,\geq\, 2\,r_B(\alpha,\beta,\tau,\lambda,\nu,d)
							\cdot q(\alpha,\beta,\tau,\lambda,\varrho,d)
				\nonumber\\
			&\,=:\, \hat{\eps}(\alpha,\beta,\tau,\lambda,\nu,\varrho,d)
				\label{eq:err>=(r0-nu*r1)*q}\,.
	\end{align}
	Fixing~\mbox{$d = d_0$}, and with appropriate values
	for the other parameters as discussed in \proofstepref{proof:monoLB3.4},
	we can provide~\mbox{$r_B(\ldots) > 0$}.
	The value of $\varrho$~should be adapted for that~$q(\ldots)$ is big
	(and positive in the first place).
	The function \mbox{$\hat{\eps}(\ldots,d)$} is monotonically
	increasing in~$d$, so an error bound for~\mbox{$d=d_0$}
	implies error bounds for~\mbox{$d \geq d_0$} while keeping in particular
	$\nu$ and~$\tau$.
	Clearly, for any~\mbox{$0 < \eps_0 < \hat{\eps}(\ldots)$},
	this gives lower bounds for the information complexity,
	\begin{equation*}
		n^{\ran}(\eps_0,F_{\mon}^d)
			\,>\, \nu \, \exp(\sigma \sqrt{d}) \,,
		\qquad\text{for $d \geq d_0$.}
	\end{equation*}
	
	
	\proofstep{proof:monoLB6}{%
		Smaller~$\eps$ and bigger exponent~$\tau$ for higher dimensions.}

The following sophisticated considerations lead to results of a new quality
compared to Blum et al.~\cite{BBL98}.
	If we have a lower bound~\mbox{%
		$\hat{\eps}(\alpha_0,\beta_0,\tau_0,\lambda,\nu,\varrho,d_0)
			> \eps_0$},
	then for~\mbox{$d \geq d_0$}
	and \mbox{$\tau_0 \leq \tau \leq \tau_0 \sqrt{d/d_0}$}
	we obtain the lower bound
	\begin{equation} \label{eq:eps(tau)}
		\hat{\eps}(\alpha(\tau),\beta(\tau),\tau,\lambda,\nu,\varrho,d)
			\,>\, \eps_0 \, \frac{\tau_0}{\tau}
			\,=:\, \eps
	\end{equation}
	with~\mbox{$\alpha(\tau) = \alpha_0 \, \frac{\tau_0}{\tau}$}
	and \mbox{$\beta(\tau) = \beta_0 \, \frac{\tau_0}{\tau}$},
	supposing the additional conditions~\mbox{$\beta_0 \leq \tau_0$}
	and \mbox{$-\tau_0 \leq \alpha_0 \leq 0$}.
	This provides the estimate
	\begin{equation*}
		n^{\ran}(\eps,F_{\mon}^d)
			\,\geq\, \nu \, 2^{\tau \sqrt{d}}
			\,=\, \nu \, 2^{\tau_0 \, \eps_0 \, \sqrt{d} / \eps} \,,
	\end{equation*}
	valid under the constraint
	$\eps_0 \, \sqrt{d_0/d} \leq \eps \leq \eps_0$.
	This is the theorem with~$c = \tau_0 \eps_0 \log 2$.
	
	In detail,
	showing~$\eqref{eq:eps(tau)}$ can be split into proving
	inequalities for the factors of~$\hat{\eps}(\ldots)$
	as defined in~\eqref{eq:err>=(r0-nu*r1)*q}, namely
	\begin{align}
		q(\alpha(\tau),\beta(\tau),\tau,\lambda,\varrho,d)
			&\,\geq\, q(\alpha_0,\beta_0,\tau_0,\lambda,\varrho,d_0) \,,
				\label{eq:q(tau)}
		\qquad \text{and} \\
		r_B(\alpha(\tau),\beta(\tau),\tau,\lambda,\nu,d)
			&\,\geq\, \frac{\tau_0}{\tau} \, r_B(\alpha_0,\beta_0,\tau_0,\lambda,\nu,d_0)
						\,.
				\label{eq:r_B(tau)}
	\end{align}
	Both factors contain the term~$\sigma_{\alpha\beta\tau}(d)$
	defined in~\eqref{eq:sigma_abt(d)}.
	With the given choice of~$\alpha(\tau)$ and $\beta(\tau)$,
	the product~$(\beta-\alpha)\tau = (\beta_0 - \alpha_0)\tau_0$
	is kept constant, which is the key element for the estimate
	\begin{equation} \label{eq:sigma(tau)}
		\sigma_{\alpha\beta\tau}(d)
			\,\leq\, \sigma_{\alpha_0,\beta_0,\tau_0}(d_0) \,.
	\end{equation}
	Here we also need
	\begin{equation*}
		1 \,\leq\, \kappa_{\alpha\tau}(d)
			\,=\, \left(1 + \frac{\alpha_0 \frac{\tau_0}{\tau} - 2\tau}{\sqrt{d}}
				\right)^{-1}
			\,\leq\, \left(1 + \frac{\alpha_0 - 2 \tau_0}{\sqrt{d_0}}\right)^{-1}
			\,=\, \kappa_{\alpha_0,\tau_0}(d_0) \,,
	\end{equation*}
	as well as
	\begin{align*}
		0 \,\leq\, K_{\alpha\beta\tau}(d)
			&\,=\, \frac{(\beta_0-\alpha_0) \frac{\tau_0}{\tau}
							}{\sqrt{d} + \alpha_0 \frac{\tau_0}{\tau} - 2\tau}
			\,\leq\, \frac{\beta_0-\alpha_0
								}{\sqrt{d} + (\alpha_0 - 2\tau_0)\sqrt{d/d_0}} \\
			&\,\leq\, \frac{\beta_0-\alpha_0
								}{\sqrt{d_0} + \alpha_0 - 2\tau_0}
			\,=\, K_{\alpha_0,\beta_0,\tau_0}(d_0) \,,
	\end{align*}
	where we used~$\tau_0/\tau \leq 1 \leq \sqrt{d/d_0}$
	combined with~$\alpha_0 \leq 0$,
	and $\tau \leq \tau_0\sqrt{d/d_0}$.
	
	Showing~\eqref{eq:q(tau)},
	by definition of~$q(\ldots)$ in \eqref{eq:def_q},
	means examining
	\mbox{$q_0(\alpha(\tau),\beta(\tau),\tau,\varrho,d)$},
	see~\eqref{eq:q_0}.
	Knowing \eqref{eq:sigma(tau)}, the remaining consideration is
	\begin{align*}
		\gamma_{\alpha\tau}(d)
			&\,=\, \left(\frac{\sqrt{d}+\alpha_0 \frac{\tau_0}{\tau}
										}{2\left(\tau + 1 /\sqrt{d}\right)}
				\right)^{\tau \sqrt{d}}
			\,\geq\, \left(\frac{\sqrt{d}+\alpha_0 \sqrt{d/d_0}
												}{2\left(\tau_0 + \sqrt{d_0} / d\right) \sqrt{d/d_0}}
						\right)^{\tau_0 \sqrt{d_0}} \\
			&\,\geq\, \left(\frac{\sqrt{d_0}+\alpha_0
											}{2\left(\tau_0 + 1 /\sqrt{d_0}\right)}
						\right)^{\tau_0 \sqrt{d_0}}
			\,=\, \gamma_{\alpha_0,\tau_0}(d_0)
			\,\geq\, 1 \,,
	\end{align*}
	once more using $\tau_0/\tau \leq 1 \leq \sqrt{d/d_0}$
	combined with~$\alpha_0 \leq 0$,
	and $\tau \leq \tau_0\sqrt{d/d_0}$.
	
	
	Showing~\eqref{eq:r_B(tau)} is more complicated,
	in view of the definition of~$r_B(\ldots)$ in~\eqref{eq:|B|>=r0-nu*r1},
	we need estimates on~$C_{\alpha\beta}$, $C_{\alpha\beta\tau}$
	and~$\kappa_{\tau}(d)$.
	The easiest part is the correcting
	factor~\mbox{$\kappa_{\tau}(d)$}, see~\eqref{eq:k_t(d)},
	for which by virtue of $\tau \leq \tau_0 \sqrt{d/d_0}$ and $d \geq d_0$
	we have
	\begin{equation*}
		1 \,\leq\, \kappa_{\tau}(d)
			\,=\, \left(1-\tau/\sqrt{d}-1/d \right)^{-1/2}
			\,\leq\, \left(1-\tau_0/\sqrt{d_0}-1/d_0\right)^{-1/2}
			\,=\, \kappa_{\tau_0}(d_0) \,.
	\end{equation*}
	For the other terms we need to take a detailed look at Gaussian integrals.
	First we have
	\begin{align*}
		C_{\alpha\beta}
			&\,=\, \frac{1}{\sqrt{2\pi}}
						\int_{\alpha}^{\beta}
								\exp\left(-\frac{x^2}{2}\right)
							\dint x\\
		[x = {\textstyle \frac{\tau_0}{\tau}} \, u]\quad
			&\,=\, \frac{\tau_0}{\tau \, \sqrt{2\pi}}
						\int_{\alpha_0}^{\beta_0}
								\exp\left(- \left(\frac{\tau_0}{\tau}\right)^2
														\, \frac{u^2}{2}
										\right)
							\dint u\\
		[\tau \geq \tau_0]\quad
			&\,\geq\, \frac{\tau_0}{\tau \, \sqrt{2\pi}}
							\int_{\alpha_0}^{\beta_0}
									\exp\left( - \frac{u^2}{2}\right)
								\dint u
			\,=\, \frac{\tau_0}{\tau} \, C_{\alpha_0, \beta_0} \,.
	\end{align*}
	The second Gaussian integral is a bit trickier,
	\begin{align*}
		C_{\alpha\beta\tau}
			&\,=\, \frac{1}{\sqrt{2\pi}}
						\int_{\alpha-\tau}^{\beta-\tau}
								\exp\left(-\frac{x^2}{2}\right)
							\dint x\\
		[\text{subst.\ }x + \tau = {\textstyle \frac{\tau_0}{\tau}} \, (u + \tau_0)]\quad
			&\,=\, \frac{\tau_0}{\tau \, \sqrt{2\pi}}
						\int_{\alpha_0-\tau_0}^{\beta_0-\tau_0}
								\exp\left(- \frac{1}{2} \,
															\left(\frac{\tau_0}{\tau}
																			\, (u + \tau_0)
																		- \tau
															\right)^2
										\right)
							\dint u\\
		[\textstyle{\frac{\tau_0}{\tau} \, (u + \tau_0) - \tau} \leq u \leq 0]\quad
			&\,\leq\, \frac{\tau_0}{\tau \, \sqrt{2\pi}}
							\int_{\alpha_0-\tau_0}^{\beta_0-\tau_0}
									\exp\left( - \frac{u^2}{2}\right)
								\dint u
			\,=\, \frac{\tau_0}{\tau} \, C_{\alpha_0, \beta_0, \tau_0} \,.
	\end{align*}
	Here, \mbox{$u \leq 0$} followed
	from the the upper integral boundary~\mbox{$u \leq \beta_0-\tau_0$}
	and the assumption~\mbox{$\beta_0 \leq \tau_0$}.
	The other constraint,
	\mbox{$\psi(\tau) := \frac{\tau_0}{\tau} \, (u + \tau_0) - \tau
						\leq u$},
	followed from \mbox{$\psi(\tau_0) = u$}
	and the monotonous decay of~\mbox{$\psi(\tau)$}
	for~\mbox{$\tau \geq \tau_0$}:
	\begin{align*}
		\psi'(\tau)
			&\,=\, -\frac{\tau_0}{\tau^2} (u + \tau_0) - 1 \\
		[\alpha_0-\tau_0 \leq u]\qquad
		 &\,\leq\,
				-\frac{\tau_0}{\tau^2}\, \alpha_0 - 1 \\
		[\alpha_0 \geq -\tau_0]\qquad
			&\,\leq\,
				\frac{\tau_0^2}{\tau^2} - 1 \\
		[\tau \geq \tau_0]\qquad
			&\,\leq\, 0 \,.
	\end{align*}
	Indeed, these estimates on $\kappa_{\tau}(d)$, $C_{\alpha\beta}$,
	and $C_{\alpha\beta\tau}$, together with the condition $d \geq d_0$,
	prove~$\eqref{eq:r_B(tau)}$.
	
\proofstep{proof:monoLB7}{%
		Example for numerical values.}
	The stated numerical values result from the setting
	\mbox{$\alpha_0 = -0.33794$}, \mbox{$\beta_0 = 0.46332$},
	\mbox{$\tau_0 = 1.47566 > \frac{1}{\log 2}$} and \mbox{$\lambda = 0.77399$}.
	We adapt \mbox{$\varrho = 0.25960$},
	and for starting dimension~\mbox{$d_0 = 100$}
	and information budget \mbox{$n_0 = 108$}
	(choosing~$\nu = n_0 \cdot 2^{-\tau_0 \sqrt{d_0}}$ accordingly)
	we obtain the lower error bound \mbox{$\hat{\eps}(\ldots)
																					= 0.0666667...
																					> \frac{1}{15} =: \eps_0$}.
\qed
\vspace{\baselineskip}
	
One might try to find different values for different~$d_0$ and $n_0$,
but since reasonable lower bounds start with~$d_0 = 100$
while implementation of algorithms seems hopeless in that dimension,
the result should be seen as rather theoretic.

\section*{Acknowledgements}

This paper is based on a chapter of the author's PhD thesis~\cite{Ku17},
which has been financially supported by the DFG Research Training Group 1523.
I wish to express my deepest gratitude to my PhD supervisor Erich Novak
for his advice during my work on these results.
I also wish to acknowledge valuable hints to literature by Mario Ullrich.


\begin{thebibliography}{99.}
\bibliographystyle{plain}
\addcontentsline{toc}{chapter}{Bibliography}

\bibitem{AA99}
	\textsc{G.~Alberti, L.~Ambrosio}.
	\newblock A geometrical approach to monotone functions in~$\R^n$.
	\newblock {\em Mathematische Zeitschrift}, 230:259--316, 1999.


\bibitem{Bakh59}
	\textsc{N.~S.~Bakhvalov}.
	\newblock On the approximate calculation of multiple integrals.
	\newblock {\em Vestnik MGU, Ser. Math. Mech. Astron. Phys. Chem.} 4:3--18,
		1959, in Russian.
	\newblock English translation: {\em Journal of Complexity}, 31(4):502--516, 2015.


\bibitem{BBL98}
	\textsc{A.~Blum, C.~Burch, J.~Langford}.
	\newblock On learning monotone Boolean functions.
	\newblock {\em Proceedings of the 39th Annual Symposium
		on Foundations of Computer Science}, 408--415, 1998.

\bibitem{BT96}
	\textsc{N.~H.~Bshouty, C.~Tamon}.
	\newblock On the Fourier spectrum of monotone functions.
	\newblock {\em Journal of the ACM}, 43(4):747-770, 1996.

\bibitem{CDV09}
	\textsc{K.~Chandrasekaran, A.~Deshpande, S.~Vempala}.
	\newblock Sampling $s$-concave functions: the limit of convexity based isoperimetry.
	\newblock {\em Approximation, Randomization, and Combinatorial Optimization.
		Algorithms and Techniques. Lecture Notes in Computer Science},
		5687:420--433, 2009.


\bibitem{GW07}
	\textsc{F.~Gao, J.~A.~Wellner}.
	\newblock Entropy estimate for high-dimensional monotonic functions.
	\newblock {\em Journal of Multivariate Analysis}, 98:1751--1764, 2007.


\bibitem{HeM11}
	\textsc{S.~Heinrich, B.~Milla}.
	\newblock The randomized complexity of indefinite integration.
	\newblock {\em Journal of Complexity}, 27(3--4):352--382, 2011.

\bibitem{HNW11}
	\textsc{A.~Hinrichs, E.~Novak, H.~Wo\'{z}niakowski}.
	\newblock The curse of dimensionality for the class of monotone functions and
	for the class of convex functions.
	\newblock {\em Journal of Approximation Theory}, 163(8):955--965, 2011.


\bibitem{KNP96}
	\textsc{C.~Katscher, E.~Novak, K.~Petras}.
	\newblock Quadrature formulas for multivariate convex functions.
	\newblock {\em Journal of Complexity}, 12(1):5--16, 1996.

%

\bibitem{Kop98}
	\textsc{K.~A.~Kopotun}.
	\newblock Approximation of $k$-monotone functions.
	\newblock {\em Journal of Approximation Theory}, 94:481--493, 1998.



\bibitem{Ku17}
	{\sc R.J.~Kunsch}.
	\newblock {\em High-Dimensional Function Approximation:
		Breaking the Curse with Monte Carlo Methods}.
	\newblock PhD Thesis, FSU Jena, 2017, available on arXiv:1704.08213 [math.NA].

\bibitem{No92mon}
	\textsc{E.~Novak}.
	\newblock Quadrature formulas for monotone functions.
	\newblock {\em Proceedings of the American Mathematical Society},
		115(1):59--68, 1992.

\bibitem{NP94}
	\textsc{E.~Novak, K.~Petras}.
	\newblock Optimal stochastic quadrature formulas for convex functions.
	\newblock {\em BIT Numerical Mathematics}, 34(2):288--294, 1994.

\bibitem{NW08}
	\textsc{E.~Novak, H.~Wo\'zniakowski}.
	\newblock {\em Tractability of Multivariate Problems},	Volume~I,
		Linear Information.
	\newblock European Mathematical Society, 2008. 

%


\bibitem{Papa93}
	\textsc{A.~Papageorgiou}.
	\newblock Integration of monotone functions of several variables.
	\newblock {\em Journal of Complexity}, 9(2):252--268, 1993.


\bibitem{RM51}
	\textsc{H.~Robbins, S.~Monro}.
	\newblock A stochastic approximation method.
	\newblock {\em The Annals of Mathematical Statistics}, 22(3):400--407, 1951.

\bibitem{Shev11}
	\textsc{I.~Shevtsova}.
	\newblock On the absolute constants in the Berry-Esseen type
		inequalities for identically distributed summands.
	\newblock Available on arXiv:1111.6554 [math.PR], 2011.
	
\bibitem{TWW88}
	\textsc{J.~F.~Traub, G.~W.~Wasilkowski, H.~Wo\'zniakowski}.
	\newblock {\em Information-Based Complexity}.
	\newblock Academic Press, 1988. 

\end{thebibliography}
\end{document}